\documentclass[10pt]{article}
\pdfoutput=1
\usepackage{setspace}
\usepackage{float}
\usepackage[english]{isodate}

\newcommand\blfootnote[1]{%
	\begingroup
	\renewcommand\thefootnote{}\footnote{#1}%
	\addtocounter{footnote}{-1}%
	\endgroup
}

\usepackage{amsthm}
\usepackage{enumerate}
\theoremstyle{plain}

\usepackage[hyperfootnotes=false]{hyperref}
\usepackage{amsmath}
\usepackage{latexsym,color}
\usepackage[dvipsnames]{xcolor}
\usepackage{amsmath}
\usepackage{amssymb}
\usepackage{amsthm}
\usepackage{amscd}
\usepackage{cite}
\usepackage{graphicx}

\usepackage{amstext}
\usepackage{array}
\usepackage{ytableau}

\usepackage{tikz}
\usepackage{comment}
\usepackage[english]{isodate}

\newcommand{\ncom}{\newcommand }
\ncom{\lmp}{\lambda^{-+}}
\ncom{\lp}{\lambda^{+}}
\ncom{\lpn}{\lambda^{+,n}}
\ncom{\lm}{\lambda^{-}}
\ncom{\lmeq}{\lambda^{-,=}}

\DeclareMathOperator{\End}{End}

\DeclareMathOperator{\Ind}{Ind}
\DeclareMathOperator{\Res}{Res}
\DeclareMathOperator{\Hom}{Hom}
\newcommand{\CC}{\mathbb{C}}
\newcommand{\ZZ}{\mathbb{Z}}
\newcommand{\ct}{{\rm{ct}}}
\newtheorem{theorem}{Theorem}[section]
\newtheorem{proposition}[theorem]{Proposition}
\newtheorem{corollary}[theorem]{Corollary}
\newtheorem{lemma}[theorem]{Lemma}

\theoremstyle{definition}
\newtheorem{definition}[theorem]{Definition}
\newtheorem{example}[theorem]{Example}
\newtheorem{remark}[theorem]{Remark}
\newtheorem{convention}[theorem]{Convention}

\hypersetup{
	colorlinks=true,
	linkcolor=blue,
	filecolor=magenta,      
	urlcolor=blue,
	citecolor=BrickRed,
	linktoc=page
}

\usepackage[titles]{tocloft}

\title{Jucys--Murphy elements of partition algebras for the rook monoid}
\author{Ashish Mishra and Shraddha Srivastava}

\newcommand{\Addresses}{{
  \bigskip
  \footnotesize

  \textsc{Ashish Mishra, Faculdade de Matem\'{a}tica, Universidade Federal do Par\'{a}, Bel\'{e}m, Brazil}\par\nopagebreak
  \textit{E-mail address}: \texttt{ashishmsr84@gmail.com}

  \medskip

  \textsc{Shraddha Srivastava, Tata Institute of Fundamental Research, Mumbai, India}\par\nopagebreak
  \textit{E-mail address}: \texttt{maths.shraddha@gmail.com}

}}

\begin{document}
\date{}
\maketitle
\begin{abstract}
Kudryavtseva and Mazorchuk exhibited Schur--Weyl duality between the rook monoid algebra $\CC R_n$ and the subalgebra $\CC I_k$ of the partition algebra $\CC A_k(n)$  acting on $(\CC^n)^{\otimes k}$. In this paper, we consider a subalgebra $\CC I_{k+\frac{1}{2}}$ of $\CC I_{k+1}$ such that there is Schur--Weyl duality between the actions of $\CC R_{n-1}$ and $\CC I_{k+\frac{1}{2}}$ on $(\CC^n)^{\otimes k}$.  This paper studies the representation theory of partition algebras $\CC I_k$ and $\CC I_{k+\frac{1}{2}}$ for rook monoids inductively by considering the multiplicity free tower $$\CC I_1\subset \CC I_{\frac{3}{2}}\subset \CC I_{2}\subset \cdots\subset \CC I_{k}\subset \CC I_{k+\frac{1}{2}}\subset\cdots.$$ Furthermore, this inductive approach is established as a spectral approach by describing the Jucys--Murphy elements and their actions on the canonical Gelfand--Tsetlin bases, determined by the aforementioned multiplicity free tower, of irreducible representations of $\CC I_k$ and $\CC I_{k+\frac{1}{2}}$. Also, we describe the Jucys--Murphy elements of $\CC R_n$ which play a central role in the demonstration of the actions of Jucys--Murphy elements of $\CC I_k$ and $\CC I_{k+\frac{1}{2}}$. 
\end{abstract}

\blfootnote{{\bf 2010 MSC}: 05E10, 20M30, 20C15.
	
	{\bf Keywords}: Rook monoid, totally propagating partition algebras, Jucys--Murphy elements, Gelfand--Tsetlin basis, RSK.}


 \section{Introduction}\label{sec:intro}
  
For $k\in \ZZ_{\geq 0}$, the set of nonnegative integers, and $\xi\in \CC$, the field of complex numbers, the partition algebra $\CC A_k(\xi)$, defined independently by Jones~\cite{Jones} and Martin~\cite{M94}, has a basis $A_k$ consisting of partition diagrams corresponding to the set partitions of $\{1,2,\ldots,k,1',2',\ldots,k'\}$. Martin and Rollet~\cite{MR98} introduced a subalgebra, denoted by $\CC A_{k+\frac{1}{2}}(\xi)$, of $\CC A_{k+1}(\xi)$. For $n\in \ZZ_{> 0}$, the set of positive integers, the partition algebras $\CC A_k(n)$ and $\CC A_{k+\frac{1}{2}}(n)$ are in Schur--Weyl duality with symmetric groups $S_n$ and $S_{n-1}$, respectively, while acting on $(\CC^n)^{\otimes k}$ \cite{Jones, MR98}. Moreover, for $l\in\frac{1}{2}\ZZ_{> 0}$, the algebra $\CC A_l(\xi)$ is semisimple unless $\xi$ is a nonnegative integer less than $2l-1$.
  Whenever partition algebras are semisimple, the branching rule for the inclusion $ \CC A_{l-\frac{1}{2}}(\xi) \subset \CC A_{l}(\xi)$  was determined in~\cite[Proposition 1]{MR98} and the Bratteli diagram for the tower of algebras
 \begin{equation}\label{eq:tower1}
\CC A_{0}(\xi) \subseteq\CC A_{\frac{1}{2}}(\xi)\subset \CC A_{1}(\xi)\subset \CC A_{\frac{3}{2}}(\xi)\subset \CC A_{2}(\xi)\subset\cdots
 \end{equation}
  was constructed. The Bratteli diagram for the tower~\eqref{eq:tower1} is a simple graph, i.e., the restriction of an irreducible representation of $\CC A_{l}(\xi)$ to $\CC A_{l-\frac{1}{2}}(\xi)$ is multiplicity free.
  
Given a tower of finite-dimensional semisimple associative algebras 
\begin{align}\label{al:tower}
\CC \cong \mathbb{A}_0 \subset \mathbb{A}_1 \subset \cdots \subset \mathbb{A}_h \subset \cdots
\end{align}
 such that the corresponding Bratteli diagram is a simple graph, there exists a canonical basis, called Gelfand--Tsetlin basis, of an irreducible representation of $\mathbb{A}_h$ (for details, see \cite[p. 585]{vo1}). A Gelfand--Tsetlin vector of $\mathbb{A}_h$ is an element of the Gelfand--Tsetlin basis of some irreducible representation of $\mathbb{A}_h$. Along with Gelfand--Tsetlin basis, the Jucys--Murphy elements play an important role in the spectral approach to the representation theory when the Bratteli diagram for the tower of groups or algebras is a simple graph; to mention some specific instances:
 for symmetric groups by Okounkov and Vershik~\cite{vo1,vo}, for generalized symmetric groups by  Mishra and Srinivasan~\cite[Section 4]{MS16}, for partition algebras by Halverson and Ram~\cite[p. 898]{HR05},  for $q$-rook monoid algebras (when $q\neq 1$) by Halverson~\cite[Section 3.1]{Hal04}, and for partition algebras for complex reflection groups by Mishra and Srivastava~\cite[Section 7]{MS}. Corresponding to the tower of algebras~\eqref{al:tower}, the Jucys--Murphy elements of $\mathbb{A}_h$ have the following fundamental properties: (i) these elements commute with each other, (ii) the Gelfand--Tsetlin basis of an irreducible representation of $\mathbb{A}_h$ is uniquely determined by the eigenvalues of the actions of these elements, and these eigenvalues distinguish the nonisomorphic irreducible representations of $\mathbb{A}_h$, and (iii) the sum of all the Jucys--Murphy elements of $\mathbb{A}_h$ is a central element of $\mathbb{A}_h$.

 Let $R_n$ be the set consisting of $n\times n$ matrices whose entries are either $0$ or $1$ such that each row and each column has at most one nonzero entry. The set $R_n$ is a monoid with respect to matrix multiplication, and Solomon~\cite{Solomon} called $R_n$ a rook monoid. Munn~\cite{Munn} characterized the irreducible representations of the rook monoid algebra $\CC R_n$ in terms of the irreducible representations of the group algebra $\CC S_r$ of the symmetric group $S_r$ for $0\leq r\leq n$. The algebra $\CC R_n$ is semisimple over $\CC$ and its irreducible representations are indexed by the partitions of $r$, where $0\leq r\leq n$. Grood~\cite{Grood} constructed an analog of the Specht modules for $\CC R_n$. 
 
 Solomon~\cite{Solomon1} defined the $q$-rook monoid algebra $\mathcal{I}_n(q)$ over $\CC(q)$, and when $q= 1$, it is the rook monoid algebra $\CC R_n$. Halverson~\cite{Hal04} constructed the irreducible representations and defined Jucys--Murphy elements of $\mathcal{I}_n(q)$. However, in~\cite[Section 3.1]{Hal04} the actions of the Jucys--Murphy elements on Gelfand--Tsetlin vectors  are not sufficient to distinguish the nonisomorphic irreducible representations of $\mathcal{I}_n(1)$. In this paper, the first important result is to close this gap for $q=1$, i.e., we give the Jucys--Murphy elements of $\CC R_n$.

  Kudryavtseva and Mazorchuk~\cite{KM08} gave the Schur--Weyl duality between the actions of $\CC R_n$ and a subalgebra of $\CC A_k(n)$ on $(\CC^n)^{\otimes k}$. This subalgebra, denoted by $\CC I_k$, is the monoid algebra of the submonoid $I_k$ of $A_k$, where $I_k$ consists of those partition diagrams each of whose block is propagating, i.e., a block which intersects nontrivially with both $\{1,2,\ldots,k\}$ and 
 $\{1',2',\ldots,k'\}$. The monoid $I_k$ appeared first time in a work of Fitzgerald and Leech~\cite{FL} as a dual symmetric inverse monoid and as a categorical dual of $R_n$. Maltcev~\cite{Mal07} described a generating set of 
$I_k$ and also determined the automorphisms of $I_k$. Easdown, East, and Fitzgerald~\cite{EEF08} gave a monoid presentation of $I_{k}$.

In this paper, we study the representation theory of monoid algebras $\CC I_k$ and $\CC I_{k+\frac{1}{2}}:=\CC I_{k+1}\cap \CC A_{k+\frac{1}{2}}(\xi)$. These algebras are independent of the parameter $\xi$ and are always semisimple over $\CC$.  We call $\CC I_{k}$ and $\CC I_{k+\frac{1}{2}}$  totally propagating partition algebras. In the light of Schur--Weyl dualities of $\CC I_k$ and $\CC I_{k+\frac{1}{2}}$ with rook monoid algebras $\CC R_n$ and $\CC R_{n-1}$, respectively, we also call $\CC I_k$ and $\CC I_{k+\frac{1}{2}}$ partition algebras of rook monoids (see Section~\ref{sec:SWD}).

 The main results in this paper are as follows.
 \begin{enumerate}[(a)]
 	\item For the rook monoid algebra:
 	\begin{enumerate}[(i)]
 		\item We give  Jucys--Murphy elements~\eqref{eq:New_JM} of $\CC R_n$ and  describe their actions on the Gelfand--Tsetlin bases of the irreducible 
 		representations of $\CC R_n$ in Theorem~\ref{thm:action}.
 		
		\item In Theorem~\ref{prop:charc}, we compute the Kronecker product of $\CC^n$ with any irreducible representation of $\CC R_n$. 

 		 	 \item  As an application of Theorem~\ref{prop:charc} and Robinson--Schensted--Knuth row-insertion algorithm, we give a bijective proof of \cite[Example 3.18]{Solomon} in Theorem~\ref{coro:bij}
 		  which gives the multiplicity of an irreducible representation of $\CC R_n$ in $(\CC^n)^{\otimes k}$ involving a Stirling number of the second kind and the number of standard Young tableaux. 
 	  \end{enumerate}
 
 	\item For totally propagating partition algebras:
 
 	\begin{enumerate}[(i)]
 		\item We construct a tower~\eqref{eq:tower} of totally propagating partition algebras using the embedding \eqref{eq:embedding} $\CC I_k\subset\CC I_{k+\frac{1}{2}}$. Theorem~\ref{thm:item2} proves that the Bratteli diagram for the tower~\eqref{eq:tower} is a simple graph.
 		
 		\item  We give Jucys--Murphy elements~\eqref{eq:JM_tppa} of $\CC I_k$ and $\CC I_{k+\frac{1}{2}}$, and describe their actions on the Gelfand--Tsetlin bases of irreducible representations of $\CC I_k$ and $\CC I_{k+\frac{1}{2}}$ in Theorem~\ref{thm:JM}. Moreover, Corollary \ref{cor:jmtppa} observes how the Jucys--Murphy elements give a spectral approach to the representation theory of totally propagating partition algebras. 
 	\end{enumerate}
 \end{enumerate}
 
  The outline of this paper is as follows. We start Section~\ref{sec:rook} with a brief overview of the irreducible representations of $\CC R_n$. 
  In the rest of this section we give new results about the representation theory of $\CC R_n$, in particular, Jucys--Murphy elements of $\CC R_n$ (Section~\ref{sec:JM_rook}), and the Kronecker product of $\CC^n$ with an irreducible representation of $\CC R_n$ (Section~\ref{sec:Kro}). Section~\ref{sec:ind_res} contains a proof of Frobenius reciprocity between modified induction and restriction rules for the inclusion $\CC R_{n-1}\subset\CC R_n$.
  
 Sections \ref{sec:partition} and \ref{sec:SWD} include the preliminaries on partition algebras and Schur--Weyl dualities, respectively. We define totally propagating partition algebras in Definition~\ref{def:total_par}; Section \ref{sec:irrep_tppa} contains a parametrization of their irreducible representations, a construction of a tower of totally propagating partition algebras and the Bratteli diagram for this tower.   
 Section~\ref{sec:JM} details about Jucys--Murphy elements of totally propagating partition algebras.
  
\section{The rook monoid algebra}\label{sec:rook}
Let $R_n$ be the set of all $n\times n$ matrices whose entries are either $0$ or $1$ such that there is at most one nonzero entry in each row and each column. Under the matrix multiplication $R_n$ is a monoid, called the rook monoid. The algebra $\CC R_n$ is called the rook monoid algebra. Section~\ref{sec:irrep_rook} contains a preliminary on representation theory of $\CC R_n$.

\subsection{The irreducible representations of \texorpdfstring{$\CC R_n$}{}}\label{sec:irrep_rook}
 In this section,  we give a well-known set of generators and relations of $\CC R_n$ as well as a construction of the irreducible representations of $\CC R_n$.

{\bf Generators and relations}.  For a transposition $(l_1,l_2)\in S_n$, its corresponding permutation matrix in $R_n$ is also denoted by $(l_1,l_2)$.  For $1\leq i\leq n-1$, let $s_i$ denote $(i,i+1)$. Let $id$ denote the $n\times n$ identity matrix. For $1\leq j\leq n$, let $P_j\in R_n$ be the diagonal matrix whose first $j$ diagonal entries are $0$ and the remaining diagonal entries are $1$.  From~\cite[Section 2]{Hal04}, 
 the set $\{s_i,P_j\mid 1\leq i\leq n-1 \text{ and } 1\leq j\leq n\}$ generates $\CC R_n$.
 Recursively using the relation 
 \begin{equation}\label{eq:P_j}
 P_j=P_{j-1}s_{j-1} P_{j-1} \text{ for } 2\leq j\leq n
 \end{equation}
 from~\cite[Lemma 1.4]{Hal04}, we see that the set
 \begin{equation}\label{eq:rgen_set}
 \{s_i,P_1\mid 1\leq i\leq n-1\}
 \end{equation}
 is also a generating set of $\CC R_n$. These generators satisfy the following relations:
 
 \noindent$(a) \, s_i^2=id$, for $1\leq i \leq n-1$, \quad\,\,\,\,\,\,$(b)\, s_is_{i+1}s_{i}=s_{i+1}s_{i}s_{i+1}$, for $1\leq i\leq n-1$, \\$(c)\, s_is_j=s_{j}s_{i}$, when $|i-j|\geq 2$, \,\,\, $(d)\, s_iP_1=P_1s_{i}$, for $2\leq i\leq n-1$, and \quad $(e)\, P_1^{2}=P_1$.

{\bf  Irreducible representations}.
Let $\Lambda_{\leq n}$ denote the set of all partitions of $r$, where $0\leq r \leq n$. Since a partition can be written equivalently as a Young diagram, therefore we use the same notation $\Lambda_{\leq n}$ to denote the set of all 
Young diagrams with total number of boxes being $r$, where  $0\leq r \leq n$. It will be clear from the context whether we are considering a partition or its Young diagram. The Young diagram $\emptyset$ containing zero boxes corresponds to the unique partition of zero with zero parts. We draw Young diagrams by following the convention of writing matrices with $X$-axis running downwards and $Y$-axis running to the right.
\begin{definition}
For $\lambda\in\Lambda_{\leq n}$, an $n$-standard tableau of shape $\lambda$ is a filling of the Young diagram of shape $\lambda$ with entries 
from $\{1,\ldots,n\}$ such that entries strictly increase along each row from left to right and along each column 
 from top to bottom. 
 \end{definition}
  Let $\tau^{\lambda}_{n}$ be the set of all $n$-standard tableaux of shape $\lambda$. Fix $L\in \tau_{n}^{\lambda}$. Let $v_{L}$ denote the symbol indexed by $L$.
Define 
\begin{displaymath}
V^{\lambda}_n:=\CC\text{-span} \{v_{L}\mid L\in \tau^{\lambda}_{n}\}. 
\end{displaymath}
 If $i$ is an entry of a box in $L$, then we write $i\in L$. Define $s_iL$ as follows: if $i\in L$, then replace $i$ by $i+1$, and if $i+1\in L$, then replace $i+1$ by $i$, and the remaining entries in $L$ are fixed. Note that $s_iL\notin\tau^{\lambda}_n$ if and only if $i$ and $i+1$ appear consecutively in the same row or same column of $L$.

The content of a box $b$ with coordinates $(x,y)$ in a Young diagram is $\ct(b):=y-x$. For $\alpha\in L$, let $L(\alpha)$ denote the box in $L$ containing $\alpha$. If $\alpha, \alpha+1\in L$, then by the definition of a $n$-standard tableau, we have $\text{ct}(L(\alpha+1))\neq \text{\ct}(L(\alpha))$, and define
\begin{equation}\label{eq:scalar}
a_{L(\alpha)}=\frac{1}{\text{ct}(L(\alpha+1))-\text{ct}(L(\alpha))}.
\end{equation}
For $1\leq i\leq n-1 $, define an action of $s_i$ 
on $V^{\lambda}_n$ as follows:
\begin{align}\label{eq:al_si} 
s_{i}v_{L}=
\begin{cases}
 v_{s_{i}L} & \text{if }i\in L, i+1\notin L,\\
 v_{s_{i}L} & \text{if }i\notin L, i+1\in L,\\
 v_{L}      & \text{if }i\notin L, i+1\notin L,\\
 a_{L(i)}v_{L}+(1+a_{L(i)})v_{L'} & \text{if } i\in L, i+1\in L,
\end{cases}
\end{align}
where
\begin{align*}
v_{L'}=\begin{cases}
v_{s_{i}L} & \text{if } s_{i}L\in\tau_{n}^{\lambda},\\
0 & \text{otherwise}.
\end{cases}
\end{align*}

Define an action of $P_1$ on $V^{\lambda}_{n}$ as follows:
\begin{align}\label{eq:P_1}
 P_{1}v_{L}= \begin{cases}
                v_{L} & \text{if }1\notin L,\\
                0 & \text{otherwise}.
               \end{cases}
\end{align}
For $2 \leq j\leq n$, using~\eqref{eq:P_j}, \eqref{eq:al_si} and \eqref{eq:P_1}, we get the action $P_j$ on $V^{\lambda}_n$ as
\begin{align}\label{eq:ac_P_j}
P_{j}v_{L}= \begin{cases}
v_{L} & \text{if }1,\ldots,j\notin L,\\
0 & \text{otherwise}.
\end{cases}
\end{align}

 \begin{example}
 For $\lambda=\emptyset$, $V^{\lambda}_n$ is isomorphic to the trivial representation $\CC$ of $\CC R_n$. For $\lambda=(1)$, $V^{\lambda}_n$ is isomorphic to the defining representation $\CC^n$ of $\CC R_n$, i.e, the elements of $R_{n}$ act on $\CC^n$ by the matrix multiplication.
\end{example}

Solomon~\cite{Solomon1} defined a quantum deformation of $\CC R_n$, known as the $q$-rook monoid algebra, over $\CC(q)$. By specializing $q=1$ in~\cite[Theorem 3.2]{Hal04}, the following theorem gives a classification of the irreducible representations of $\CC R_n$.
\begin{theorem}\label{thm:irrep}
 The set $\{V^{\lambda}_n\mid \lambda\in\Lambda_{\leq n}\}$ is a complete set of pairwise nonisomorphic irreducible representations
 of $\CC R_{n}$.
\end{theorem}

The monoid $R_{n-1}$ is isomorphic to the submonoid of $R_{n}$ consisting of elements whose $(n,n)$-th entry is $1$, so $\CC R_{n-1}$ embeds inside $\CC R_{n}$. The following proposition is~\cite[Corollary 3.3]{Hal04} which describes the branching rule for the embedding $\CC R_{n-1}\subset \CC R_n$.

For $\lambda\in\Lambda_{\leq n}$, let $\lambda^{-,=}$ denote the set consisting of $\nu\in\Lambda_{\leq (n-1)}$ such that either $\nu$ is obtained by removing a box from an inner corner of $\lambda$ or $\nu=\lambda$. Similarly, for $\mu\in\Lambda_{\leq n-1}$, let $\mu^{+,=}$ denote the set 
consisting of $\nu'\in\Lambda_{\leq n}$ such that either $\nu'$ is obtained by adding a box to an outer corner of $\mu$ or $\nu'=\mu$.
For the definitions of an inner corner and an outer corner of a Young diagram, see~\cite[Definition 2.8.1]{Sagan}.
 
\begin{proposition}[Branching rule] 
\label{prop:res}
For $\lambda\in\Lambda_{\leq n}$ and $\mu\in \Lambda_{\leq n-1}$, we have 
\begin{displaymath}
 {\rm Res}^{\CC R_{n}}_{\CC R_{n-1}}(V^{\lambda}_n)\cong \underset{\nu\in \lambda^{-,=}}{\bigoplus}V^{\nu}_{n-1},~~~ \mbox{ and }~~~~
 {\rm Ind}^{\CC R_{n}}_{\CC R_{n-1}}(V^{\mu}_{n-1})\cong \underset{\nu'\in \mu^{+,=}}{\bigoplus}V^{\nu'}_{n}.
\end{displaymath}

\end{proposition}

From Proposition~\ref{prop:res}, we conclude that the Bratteli diagram for the tower of algebras 
\begin{equation}\label{eq:tower_rook}
\CC \subset \CC R_1\subset \CC R_2 \subset  \cdots \CC R_n \subset \cdots
\end{equation} is a simple graph. Figure \ref{fig:BDRM} is the Bratteli diagram for the tower~\eqref{eq:tower_rook} up to level $3$. (The definition of a Bratteli diagram can be found in~\cite[p. 584]{vo1}).
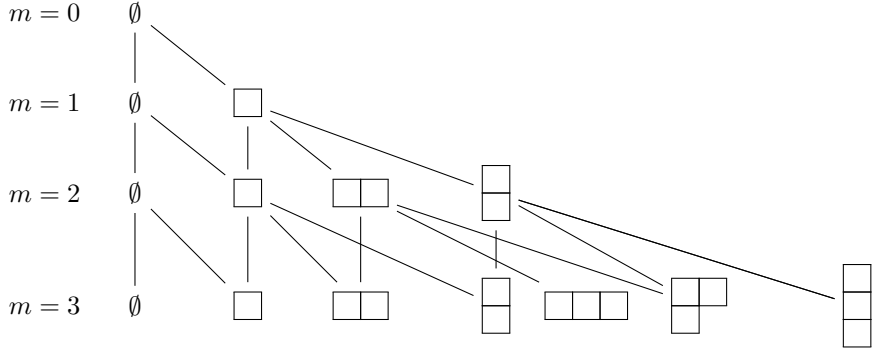
\begin{figure}[ht]
 \begin{tikzpicture}[scale=0.6]
 \node (a) at (-2,0) {$m=0$};
 \node (b) at (-2,-2) {$m=1$};
 \node (c) at (-2,-4) {$m=2$};
 \node (d) at (-2,-6.5) {$m=3$};
  \node (1) at (0,0) {$\emptyset$};
  \node (1') at (0,-2) {$\emptyset$};
  \node (2) at  (2.5,-2) {$\ytableausetup{boxsize=1em} \begin{ytableau}
                                                        \empty \\
                                                       \end{ytableau}$};
  \node (2') at (0,-4) {$\emptyset$};
  \node (3) at  (2.5,-4) {$\ytableausetup{boxsize=1em} \begin{ytableau}
                                                        \empty \\
                                                       \end{ytableau}$};
  \node (4) at (5,-4) {$\ytableausetup{boxsize=1em} \begin{ytableau}
                                                        \empty & \empty \\
                                                       \end{ytableau}$};
  \node (5) at (8,-4) {$\ytableausetup{boxsize=1em} \begin{ytableau}
                                                        \empty \\
                                                        \empty \\
                                                       \end{ytableau}$};
  \node (5') at (0,-6.5) {$\emptyset$};
  \node (6) at (2.5,-6.5) {$\ytableausetup{boxsize=1em} \begin{ytableau}
                                                        \empty \\
                                                       \end{ytableau}$};
  \node (7) at (5,-6.5) {$\ytableausetup{boxsize=1em} \begin{ytableau}
                                                        \empty & \empty \\
                                                       \end{ytableau}$};
  \node (8) at (8,-6.5) {$\ytableausetup{boxsize=1em} \begin{ytableau}
                                                        \empty \\
                                                        \empty \\
                                                       \end{ytableau}$};
  \node (9) at (10,-6.5) {$\ytableausetup{boxsize=1em} \begin{ytableau}
                                                        \empty & \empty & \empty\\
                                                       \end{ytableau}$};
  \node (10) at (12.5,-6.5) {$\ytableausetup{boxsize=1em} \begin{ytableau}
                                                        \empty & \empty \\
                                                        \empty \\
                                                       \end{ytableau}$};
  \node (11) at (16,-6.5) {$\ytableausetup{boxsize=1em} \begin{ytableau}
                                                        \empty \\
                                                        \empty \\
                                                        \empty \\
                                                       \end{ytableau}$};
  \draw (1) -- (1') -- (2') -- (5');
  \draw (1) -- (2) -- (3) -- (6);
  \draw (2) -- (4) -- (7);
  \draw (2) -- (5) -- (11);
  \draw (3) -- (7);
  \draw (3) -- (8);
  \draw (1') -- (3);
  \draw (4) -- (9);
  \draw (4) -- (10);
  \draw (2') -- (6);
  \draw (5) -- (8);
  \draw (5) -- (10);
  \draw (5) -- (11);
 \end{tikzpicture}
 \caption{Bratteli diagram, up to level $3$, for the tower of rook monoid algebras}
 \label{fig:BDRM}
 \end{figure} 

\begin{convention}\label{con}
For a graph $\mathcal{G}$ appearing in this article, a path in $\mathcal{G}$ from $\lambda$ at the level $r$ to $\mu$ at the level $s(>r)$ is a tuple of vertices $(\nu_{t_{1}},\nu_{t_2},\ldots,\nu_{t_{y}})$ at the consecutive levels $t_1,t_2,\ldots, t_y$ such that $t_1=r, t_{y}=s, \nu_{t_{1}}=\lambda$, $\nu_{t_{y}}=\mu$ and, for $1\leq i\leq y-1$, 
there is an edge between $\nu_{t_{i}}$ and  $\nu_{t_{i+1}}$.
\end{convention}

The following lemma is analogous to the correspondence between the paths in the Bratteli diagram for the tower of symmetric groups and the standard tableaux.
\begin{lemma}\label{lm:paths}
	For $\lambda\in \Lambda_{\leq n}$, a path from the vertex $\emptyset$ at the level $m=0$ to the vertex $\lambda$ at the level $m=n$ in the Bratteli diagram for the tower~\eqref{eq:tower_rook} corresponds to a unique $n$-standard tableau in  $\tau_{n}^{\lambda}$. Moreover, this correspondence is a bijection.
\end{lemma}
 
 \begin{example}
 	Using Lemma~\ref{lm:paths}, for $m=3$ and $\lambda=\ytableausetup{boxsize=0.7em}
\begin{ytableau}
\empty & \empty \\
\end{ytableau}$,  the paths $(\emptyset,\ytableausetup{boxsize=0.7em}
\begin{ytableau}
\empty \\
\end{ytableau},\ytableausetup{boxsize=0.7em}
\begin{ytableau}
\empty & \empty \\
\end{ytableau},\ytableausetup{boxsize=0.7em}
\begin{ytableau}
\empty & \empty \\
\end{ytableau}\,)$, $(\emptyset,\ytableausetup{boxsize=0.7em}
\begin{ytableau}
\empty \\
\end{ytableau},\ytableausetup{boxsize=0.7em}
\begin{ytableau}
\empty \\
\end{ytableau},\ytableausetup{boxsize=0.7em}
\begin{ytableau}
\empty & \empty \\
\end{ytableau}\,)$, and $(\emptyset,\emptyset,\ytableausetup{boxsize=0.7em}
\begin{ytableau}
\empty \\
\end{ytableau},\ytableausetup{boxsize=0.7em}
\begin{ytableau}
\empty & \empty \\
\end{ytableau}\,)$  correspond to the $3$-standard tableaux
$\ytableausetup{boxsize=0.7em}
\begin{ytableau}
\scriptstyle 1 & \scriptstyle 2 \\
\end{ytableau}$, $\ytableausetup{boxsize=0.7em}
\begin{ytableau}
\scriptstyle 1 & \scriptstyle 3 \\
\end{ytableau}$, and 
$\ytableausetup{boxsize=0.7em}
\begin{ytableau}
\scriptstyle 2 & \scriptstyle 3 \\
\end{ytableau}$, respectively.
\end{example}

\begin{remark}	For $\lambda\in\Lambda_{\leq n}$, by Theorem~\ref{thm:irrep}, Proposition~\ref{prop:res}, and Lemma~\ref{lm:paths}, we see that $\{v_{L}\mid L\in\tau_{n}^{\lambda}\}$ is the Gelfand--Tsetlin basis of the irreducible representation $V^{\lambda}_{n}$ of $\CC R_n$. See Section~\ref{sec:intro} for Gelfand--Tsetlin basis and Gelfand--Tsetlin vectors.
\end{remark}

\subsection{Jucys--Murphy elements of the rook monoid algebra}\label{sec:JM_rook}

Halverson~\cite[Section 3.1]{Hal04} defined Jucys--Murphy elements of the $q$-rook monoid algebra and described their actions on Gelfand--Tsetlin vectors \cite[Proposition 3.5]{Hal04}. It can be seen in Example~\ref{ex:JM} that when $q=1$, actions of the elements in~\cite[Section 3.1]{Hal04} on Gelfand--Tsetlin vectors do not distinguish the nonisomorphic irreducible representations of $\CC R_n$ (see Section~\ref{sec:intro} for the properties of Jucys--Murphy elements). In this section, we close this gap for $q=1$.

 For  $2 \leq i \leq n$, define $Q_i = (2, i-1)(1, i)P_2(1,i)(2,i-1)$.\label{pa:Q_i}
\begin{lemma}\label{lm:P2} 
 Let $\lambda\in\Lambda_{\leq n}$. Then for the Gelfand--Tsetlin vector $v_{L}\in V^{\lambda}_n$ corresponding to $L\in \tau^{\lambda}_{n}$ and for $2\leq i\leq n$, we have
	\begin{align}\label{al:Q}
	Q_iv_{L}=\begin{cases}
	v_{L} & \text{if } i-1,i\notin L,\\
	0 & \text{otherwise}.
	\end{cases}
	\end{align}
\end{lemma}
\begin{proof}
	
	 Observe that $Q_i\in R_n$ is the diagonal matrix whose $(i-1)$-th and 
	$i$-th diagonal entries are $0$ and the remaining diagonal entries are $1$.
	It can be seen that for $i\geq 3$, we have $Q_i=s_{i-2}s_{i-1} Q_{i-1} s_{i-1}s_{i-2}$. We use induction on $i$ to prove~\eqref{al:Q}. For $i=2$,~\eqref{al:Q} follows from~\eqref{eq:ac_P_j}. For $i\geq 3$, suppose that~\eqref{al:Q} is true for some $i-1$ and we prove it for $i$.

\noindent{ Case(i)}:  $i-1\notin L$ and $i\notin L$.

\noindent{ Subcase(a)}:  $i-2\in L$. Then $s_{i-1}s_{i-2}{L}$ is  an $n$-standard tableau which does not contain both $i-2$ and $i-1$. Using~\eqref{eq:al_si} and the induction hypothesis, $Q_{i-1} s_{i-1}s_{i-2}v_{L}=v_{s_{i-1}s_{i-2}L}$. Thus $Q_{i-1}v_{L}=v_{L}$.

\noindent{ Subcase(b)}:  $i-2\notin L$. Then $s_{i-1}s_{i-2}{L}=L$, and using~\eqref{eq:al_si}
$s_{i-1}s_{i-2}v_{L}=v_L$. Since $i-2\notin L$ and $i-1\notin L$, by the induction hypothesis we get 
$Q_{i-1}s_{i-1}s_{i-2}v_{L}=v_{L}$, and so $Q_{i}v_{L}=v_{L}$.

\noindent{ Case(ii)}: $i-1\in L$ and $i\notin L$. Then using~\eqref{eq:al_si}
$s_{i-1}s_{i-2}v_{L}$ is a linear combination of $v_{L'}$ for $L'\in\tau_{n}^{\lambda}$ with
$i-2\in L'$. So by the induction hypothesis $Q_{i-1}s_{i-1}s_{i-2}v_{L}=0$, and this implies 
$Q_iv_{L}=0$. 

The remaining cases are when $i-1\notin L$, $i\in L$ or  when $i-1\in L, i\in L$. In both these cases, by following similar reasoning as in Case(ii), we have $Q_{i-1}s_{i-1}s_{i-2}v_{L}=0$, which implies $Q_iv_{L}=0$. 
 \end{proof}

From~\cite[Section 3.1]{Hal04}, we have the following elements of $\CC R_n$:
\begin{align*} 
X_{1}&=1-P_{1},  \text{ and }\\
X_{i}&=s_{i-1}X_{i-1}s_{i-1}, \text{ for } 2\leq i\leq n.
\end{align*}

\phantomsection For $1 \leq i \leq n$, define $\gamma_i:=1-X_i\label{pa:gamma}$. Now, we define the new elements in $\CC R_n$ as follows: 
\begin{align}\label{eq:New_JM}
\tilde{X}_1&=0, \mbox{ and } \nonumber\\
\tilde{X}_{i}=s_{i-1}\tilde{X}_{i-1}s_{i-1}+s_{i-1}&-s_{i-1}\gamma_{i-1}-\gamma_{i-1}s_{i-1} + Q_i,  \mbox{ for } 2\leq i\leq n.
\end{align}
 In the following theorem, we describe the actions of elements~\eqref{eq:New_JM} on the Gelfand--Tsetlin vectors. 

\begin{theorem}\label{thm:action}
Let $\lambda\in\Lambda_{\leq n}$. Then for the Gelfand--Tsetlin vector $v_{L}\in V^{\lambda}_n$ corresponding to $L\in \tau^{\lambda}_{n}$ and for $1\leq i\leq n$, we have
\begin{align}\label{eq:jm_gamma}
X_{i}v_{L}=\begin{cases}
	v_{L} & \text{if } i\in L,\\
	0 & \text{otherwise,}
	\end{cases}
	\end{align} 
	 \begin{align}\label{eq:new_jm}
\text{and \, \,\,}\tilde{X}_{i}v_{L}=\begin{cases}
\ct(L(i)) v_{L} & \text{if } i\in L,\\
0 & \text{otherwise.}
\end{cases}
\end{align}

\end{theorem}

\begin{proof} Equation~\eqref{eq:jm_gamma} is proved in \cite[Proposition 3.5]{Hal04}. From~\eqref{eq:jm_gamma}, we have
	\begin{align}\label{eq:jm}
	\gamma_{i}v_{L}=\begin{cases}
	v_{L} & \text{if } i\notin L,\\
	0 & \text{otherwise}.
	\end{cases}
	\end{align}
	We prove \eqref{eq:new_jm} using induction on $i$. For $i=1$, the result is true as $\tilde{X}_{1}=0$ and if $1\in L$, then $\ct(L(1))=0$.
	Assume that~\eqref{eq:new_jm} holds for some $i-1$, where $i\geq 3$, and we prove \eqref{eq:new_jm} for $i$.

\noindent{ Case(i)}: $i-1\notin L$ and  $i\notin L$. Then using~\eqref{eq:al_si} $s_{i-1}v_{L}=v_{L}$ and because $i-1\notin L$ so
by the induction hypothesis, $s_{i-1}\tilde{X}_{i-1}s_{i-1}v_{L}=0$.
By~\eqref{eq:jm} $\gamma_{i-1}s_{i-1} v_{L}=v_{L}$ and also $s_{i-1}\gamma_{i-1}v_{L}=v_{L}$, by Lemma~\ref{lm:P2} $Q_{i}v_{L}=v_{L}$. Therefore $\tilde{X}_{i}v_{L}=0$.
\vspace{0.2cm}

\noindent{ Case(ii)}: $i-1\in L$ and $i\notin L$. Since $i-1\notin s_{i-1}L$ and $s_{i-1}v_{L}=v_{s_{i-1}L}$, by the induction hypothesis, we have 
$s_{i-1}\tilde{X}_{i-1}s_{i-1}v_{L}=0$. Since $i-1\in L$, by~\eqref{eq:jm}, $s_{i-1}\gamma_{i-1}v_{L}=0$, and by Lemma~\ref{lm:P2} $Q_{i}v_{L}=0$. Again by~\eqref{eq:jm}, $\gamma_{i-1}s_{i-1}v_{L}=v_{s_{i-1}L}$. Combining all of these, we have $\tilde{X}_{i}v_{L}=0$.
 \vspace{0.2cm}
 
\noindent {Case(iii)}: $i-1\notin L$ and $i\in L$. Since $i-1\in s_{i-1}L$ such that $(s_{i-1}L)(i-1)=L(i)$ and also 
$s_{i-1}v_{L}=v_{s_{i-1}L}$, by the induction hypothesis we have $s_{i-1}\tilde{X}_{i-1}s_{i-1}v_{L}=\ct(L(i))v_{L}$. By~\eqref{eq:jm}
$\gamma_{i-1}s_{i-1}v_{L}=0$ and $s_{i-1}\gamma_{i-1}v_{L}=v_{s_{i-1} L}$, and by Lemma~\ref{lm:P2} $Q_{i}v_{L}=0$. Combining all of these, we have, $\tilde{X}_{i}v_{L}=\ct({L(i)})v_{L}+v_{s_{i-1} L}-v_{s_{i-1}L}=\ct({L(i)})v_{L}$.

\vspace{0.2cm}

\noindent{Case(iv)}: $i-1\in L$ and $i\in L$. Using~\eqref{eq:al_si}, we have $s_{i-1}v_{L}=a_{L(i-1)}v_{L}+(1+a_{L(i-1)})v_{L'}$ where $a_{L(i-1)}= \frac{1}{\ct(L(i))-\ct(L(i-1))} $, and  
$v_{L'}\neq 0$ if and only if $L'=s_{i-1}L\in\tau_{n}^{\lambda}$, which contains both $i-1,i$. By~\eqref{eq:jm}, $\gamma_{i-1}s_{i-1}v_{L}=0$ and $s_{i-1}\gamma_{i-1}v_{L}=0$, and by 
Lemma~\ref{lm:P2} $Q_{i}v_{L}=0$. Then we have $\tilde{X}_{i}v_{L}=s_{i-1}\tilde{X}_{i-1}s_{i-1}v_{L}+s_{i-1}v_{L}$.
\vspace{0.2cm}

\noindent{Subcase(a)}: $s_{i-1}L\notin\tau_{n}^{\lambda}$. Then, $i-1$ and $i$ occur consecutively either in the same
row or in the same column of $L$. So $a_{L(i-1)}=\pm 1$.  
Then,
\begin{align*}
s_{i-1}\tilde{X}_{i-1}s_{i-1}v_{L}+s_{i-1}v_{L}&= s_{i-1}\tilde{X}_{i-1} (a_{L(i-1)} v_{L})+ a_{L(i-1)} v_{L}\\
&= s_{i-1}(\ct(L(i-1))a_{L(i-1)} v_{L})+ a_{L(i-1)} v_{L}\\
&=  (\ct(L(i-1))({a_{L(i-1)}})^2+ a_{L(i-1)}) v_{L},
\end{align*}
where the second equality follows from the induction hypothesis.
Now,
\begin{align*}
(\ct(L(i-1))(a_{L(i-1)})^2+ a_{L(i-1)})&= a_{L(i-1)} (\ct(L(i-1)) a_{L(i-1)}+1)\\
&= \ct(L(i))(a_{L(i-1)})^2\\
&=\ct(L(i)), \text{ as $a_{L(i-1)}=\pm 1$}.
\end{align*}
The second equality follows by substituting $a_{L(i-1)}=\frac{1}{\ct(L(i))-\ct(L(i-1))}$.

\noindent{Subcase(b)}: $s_{i-1}L\in\tau_{n}^{\lambda}$. In this case $L'=s_{i-1}L$, so $s_{i-1}L'=L$. Observe that 
\begin{displaymath}
\ct(L'(i-1))=\ct(L(i)) \text{ and } a_{L'(i-1)}=-a_{L(i-1)}.
\end{displaymath}
Then,	\begin{align*}
	& s_{i-1}\tilde{X}_{i-1}s_{i-1}v_{L}+s_{i-1}v_{L}\\&= s_{i-1}\tilde{X}_{i-1}[a_{L(i-1)} v_{L}+(1+ a_{L(i-1)})v_{L'}]+
	a_{L(i-1)} v_{L}+(1+ a_{L(i-1)}) v_{L'}\\
	&= s_{i-1}[\ct(L(i-1)) a_{L(i-1)} v_{L}+(1+ a_{L(i-1)})\ct(L(i))v_{L'}]+  a_{L(i-1)} v_{L}+(1+ a_{L(i-1)}) v_{L'}\\
	&=\ct(L(i-1))a_{L(i-1)} [a_{L(i-1)} v_{L}+(1+ a_{L(i-1)}) v_{L'}]+\\ & \quad  \quad \ct(L(i))(1+a_{L(i-1)})[-a_{L(i-1)} v_{L'}+(1-a_{L(i-1)})v_{L}]+ 
	a_{L(i-1)} v_{L}+(1+ a_{L(i-1)}) v_{L'}\\
	& = Av_{L}+Bv_{L'}, 
	\end{align*} 
where	
\begin{align*}
A&=\ct(L(i-1))a_{L(i-1)}^2+ \ct(L(i))(1+ a_{L(i-1)})(1-a_{L(i-1)})+ a_{L(i-1)}, \text{ and}\\
B&= \ct(L(i-1))a_{L(i-1)}(1+a_{L(i-1)})-\ct(L(i)) (1+ a_{L(i-1)}) a_{L(i-1)}+(1+ a_{L(i-1)}).
\end{align*}
	
Now we simplify the expressions of $A$ and $B$. We have
\begin{align*}
A&= \ct(L(i-1)) a_{L(i-1)}^2+ \ct(L(i))(1+ a_{L(i-1)})(1- a_{L(i-1)})+a_{L(i-1)}\\
&= \ct(L(i-1)) a_{L(i-1)}^2+ \ct(L(i))(1- a_{L(i-1)}^2)+a_{L(i-1)}\\
&= \ct(L(i))+ a_{L(i-1)} [ \ct(L(i-1)) a_{L(i-1)}-\ct(L(i))a_{L(i-1)} +1]\\
&= \ct(L(i))+ a_{L(i-1)} \bigg(\frac{\ct(L(i-1))}{\ct(L(i))-\ct(L(i-1))}-\frac{\ct(L(i))}{\ct(L(i))-\ct(L(i-1))}+1\bigg)=\ct (L(i)), \text{ and }\\
& \\
 B&= \ct(L(i-1))a_{L(i-1)}(1+ a_{L(i-1)})-\ct(L(i))  (1+a_{L(i-1)})  a_{L(i-1)}+(1+ a_{L(i-1)})\\
&= (1+a_{L(i-1)})\bigg(\frac{\ct(L(i-1))}{\ct(L(i))-\ct(L(i-1))}-\frac{\ct(L(i))}{\ct(L(i))-\ct(L(i-1))}+1\bigg)=0.
\end{align*}
Therefore, we have $\tilde{X}_{i}v_{L}=\ct(L(i))v_{L}$.
\end{proof}

\begin{corollary}\label{coro:commutation} For $1\leq i,j\leq n$, we have $X_i X_j=X_j X_i, \text{ } X_i \tilde{X}_j=\tilde{X}_{j} X_i, \text{ and } \tilde{X}_i\tilde{X_j}=\tilde{X}_j\tilde{X}_i.$
	\end{corollary}
\begin{proof}
	Given $\lambda \in \Lambda_{\leq n}$,
	Theorem~\ref{thm:action} says that 
 each $X_i$ and each $\tilde{X}_j$, for $1\leq i,j\leq n$,  acts diagonally on the Gelfand--Tsetlin basis $\{v_{L}\mid L\in \tau^{\lambda}_n\}$
 of the irreducible representation $V^{\lambda}_n$ of $\CC R_n$. By considering the left regular representation of $\CC R_n$, which is a faithful representation, we obtain all three commutation results as stated in the corollary. 
\end{proof}

\begin{theorem}\label{thm:a}
	The elements $\kappa_{n}:=\sum_{i=1}^{n}X_{i}$ and $\tilde{\kappa}_{n}:=\sum_{i=1}^{n}\tilde{X}_{i}$ are in the center of  $\mathbb{C}R_{n}$. 
\end{theorem}

\begin{proof}
	
It is enough to prove that $\kappa_n$ and $\tilde{\kappa}_n$	commute with the generators $P_1, s_1,s_2,\ldots,s_{n-1}$ of $\CC R_n$. 
Let $1\leq i\leq n$. By definition, $X_i=(1-\gamma_i)$, so
	$\kappa_n=n.1-\bigg(\sum_{i=1}^{n}\gamma_i\bigg)$.  
	To prove $\kappa_n$ is a central element, we show that $\bigg(\sum_{i=1}^{n}\gamma_i\bigg)$ commutes  with $P_1, s_1,s_2,\ldots,s_{n-1}$. 
	
Let $s_{0}=id$. Using induction on $i$, we have $\gamma_i=s_{i-1}s_{i-2}\cdots s_1P_1 s_1\cdots s_{i-2}s_{i-1}$, which is a diagonal matrix with $(i,i)$-th entry zero and the remanining entries one. So $P_1\gamma_i=\gamma_i P_1$ which implies $P_1\bigg(\sum_{i=1}^{n}\gamma_i\bigg)=\bigg(\sum_{i=1}^{n}\gamma_i\bigg)P_1$.  For $1\leq j< n$, we have
\begin{align}\label{al:sj}
s_j\gamma_js_j=\gamma_{j+1} \text{ and } s_j\gamma_{j+1}s_{j}=\gamma_{j}; \quad	 s_j\gamma_i s_j=\gamma_i \text{ when } i\notin\{j,j+1\}.
\end{align}	
	Thus $s_j \bigg(\sum_{i=1}^{n}\gamma_i\bigg)=\bigg(\sum_{i=1}^{n}\gamma_i\bigg) s_j$ and $\kappa_n$ is a central element of $\CC R_n$.
	
	Now we proceed towards proving that $\tilde{\kappa}_n$ is a central element of $\CC R_n$. Recall that
	\begin{align*}
	\tilde{X}_1&=0 \text{ and} \text{ for }2\leq i\leq n,
	\tilde{X}_{i}=s_{i-1}\tilde{X}_{i-1}s_{i-1}+s_{i-1}-s_{i-1}\gamma_{i-1}-\gamma_{i-1}s_{i-1}+Q_i.
	\end{align*}
	
	For $i\in\{1,2,3\}$, it is easy to verify that $P_1 \tilde{X}_{i}=\tilde{X}_{i}P_1$. By taking base step $i=3$, we use 
	induction on $i$ to prove 
	\begin{equation}\label{eq:1}
	P_1 \tilde{X}_{i}=\tilde{X}_{i}P_1 \,\, \text{ for } 3\leq i\leq n.
	\end{equation} 
	Note that 
	\begin{align}
	P_1 s_{j}&=s_{j}P_1\,\, \text{ for } 2\leq j\leq n-1\label{eq:2}, \text{ and}\\
	P_1 D&= DP_1 \,\, \text{ for any diagonal matrix $D\in R_n$}\label{eq:3}.
	\end{align}
	Suppose~\eqref{eq:1} holds for $3\leq i<n$. Then using~\eqref{eq:2}, induction hypothesis, and that $Q_i$ and $\gamma_i$ are the diagonal matrices, we get $P_{1}\tilde{X}_{n}=\tilde{X}_{n}P_{1}$. Hence, $P_1\tilde{\kappa}_n=\tilde{\kappa}_nP_1$. 
	
Note that $\tilde{\kappa}_n= \sum_{i=2}^{n}\tilde{X}_{i}$. For $2\leq i\leq n$, using induction on $i$, we can check that
	\begin{displaymath}
	\tilde{X}_{i}= \sum_{r=1}^{i-1}\bigg((r,i)+ (2,r)(1,i)P_2(1,i)(2,r)-(r,i)\gamma_r-\gamma_{r}(r,i)\bigg)
	\end{displaymath}
	and so
	\begin{align*}
	\sum_{i=2}^{n}\tilde{X}_{i}=&\sum_{i=2}^{n}\sum_{r=1}^{i-1}\bigg((r,i)+(2,r)(1,i)P_2(1,i)(2,r)
	- (r,i)\gamma_r-\gamma_{r}(r,i)\bigg).
	\end{align*}
	
	Since $\sum_{i=2}^{n}\sum_{r=1}^{i-1}(r,i)\in \CC R_n$ is the sum of all transpositions in $S_n$, we have
	\begin{displaymath}
	s_j\sum_{i=2}^{n}\sum_{r=1}^{i-1}(r,i)=\sum_{i=2}^{n}\sum_{r=1}^{i-1}(r,i) s_j \text{ for all $1\leq j< n$}.
	\end{displaymath}
	
	Let $\mathcal{L}\subset R_n$ consist of diagonal matrices with exactly two diagonal entries zero and the remaining 
	diagonal entries  one. Then
	\begin{equation*}
	\sum_{i=2}^{n}\sum_{r=1}^{i-1}(2,r)(1,i)P_2(1,i)(2,r)=\sum_{M\in\mathcal{L}}M.
	\end{equation*}
	For $1\leq j<n$, since $\{s_jMs_j\mid M\in\mathcal{L}\}=\mathcal{L}$, we have
	\begin{align*}
	s_j\bigg(\sum_{i=2}^{n}\sum_{r=1}^{i-1}(2,r)(1,i)P_2(1,i)(2,r)\bigg)s_j&=s_j\bigg(\sum_{M\in\mathcal{L}}M\bigg)s_j
	= \sum_{M\in\mathcal{L}}s_jMs_j
	=\sum_{M\in \mathcal{L}} M.
	\end{align*}
	Now in order to conclude that $\tilde{\kappa}_n$ is a central element of $\CC R_n$,  it remains to show the following for $1\leq j< n$:
	\begin{equation}\label{eq:kappa_last}
	s_j\sum_{i=2}^{n}\sum_{r=1}^{i-1}\bigg((r,i)\gamma_r+\gamma_{r}(r,i)\bigg)=\sum_{i=2}^{n}\sum_{r=1}^{i-1}\bigg((r,i)\gamma_r+\gamma_{r}(r,i)\bigg)s_j.  
	\end{equation}
	 For $2\leq i\leq n$, $1\leq j<n$ and  $1\leq r\leq i-1$, the following relations can be observed using~\eqref{al:sj}:
	\begin{alignat}{3}
	 &\text{for } i\notin\{j,j+1\}\quad  && s_j(j,i)\gamma_j=(j+1,i)\gamma_{j+1}s_j,\quad && s_j\gamma_j(j,i)=\gamma_{j+1}(j+1,i)s_j,\label{al:1}\\
	 &\text{for } i\notin\{j,j+1\}\quad && s_j(j+1,i)\gamma_{j+1}=(j,i)\gamma_js_j,\quad && s_j\gamma_{j+1} (j+1,i)=\gamma_j(j,i)s_j,\label{al:2}\\
	&\text{for } r,i\notin\{j,j+1\}\quad	&& s_j(r,i)\gamma_r=(r,i)\gamma_r s_j,\quad &&s_j\gamma_r (r,i)=\gamma_r(r,i)s_j,\label{al:3}\\
	  &\text{for } r\notin\{j,j+1\}\quad  && s_j(r,j)\gamma_r=(r,j+1)\gamma_rs_j, \quad&&  s_j(r,j+1)\gamma_r=(r,j)\gamma_rs_j,\label{al:4}\\
	 &\text{for } r\notin\{j,j+1\}\quad  && s_j\gamma_r(r,j)=\gamma_r(r,j+1)s_j, \quad && s_j\gamma_r(r,j+1)=\gamma_r(r,j)s_j.\label{al:5}
  	\end{alignat}
	For $j=1$,~\eqref{eq:kappa_last} holds using~\eqref{al:1}--\eqref{al:3}.
	For $1<j< n$, we have 
	\begin{align*}
	\sum_{i=2}^{n}\sum_{r=1}^{i-1}\bigg((r,i)\gamma_r+\gamma_{r}(r,i)\bigg)
	= &\sum_{i=2,i\notin\{j, j+1\}}^{n}\sum_{r=1}^{i-1}\bigg((r,i)\gamma_r+\gamma_{r}(r,i)\bigg)+ (j,j+1)\gamma_j+\gamma_j(j,j+1) \\
	& \,+\sum_{r=1}^{j-1}\bigg((r,j)\gamma_r +\gamma_{r}(r,j)  
	+ (r,j+1)\gamma_r + \gamma_{r}(r,j+1)\bigg),
	\end{align*}
	and now it can be seen that \eqref{eq:kappa_last} holds using~\eqref{al:1}--\eqref{al:5}.
	\end{proof}

\begin{remark}
	We observe from Theorem~\ref{thm:action}, Corollary~\ref{coro:commutation} and Theorem~\ref{thm:a} that the elements $X_{i}$, $\tilde{X}_{i}$, for $1\leq i\leq n$, play the role of Jucys--Murphy elements of $\CC R_n$. Moreover, the Gelfand--Tsetlin vector $v_{L}$ of $\CC R_n$ is completely determined by  the eigenvalues of the action of  $X_i$ together with the eigenvalues of the action of $\tilde{X}_{i}$, for all $1\leq i\leq n$, on $v_{L}$.
\end{remark}

\begin{example}\label{ex:JM}
	Let $1\leq j\leq n$.  For a partition $\lambda$ of $n$, the eigenvalue of the action of $X_{j}$  on the Gelfand--Tsetlin vector $v_{L}$, for $L\in\tau_{n}^{\lambda}$, is one. So, given distinct partitions $\lambda$ and $\mu$ of $n$, the eigenvalues of the actions of $X_{i}$, for all $1\leq i\leq n$, on Gelfand--Tsetlin vectors do not distinguish the nonisomorphic irreducible representations $V^{\lambda}_n$ and $V^{\mu}_n$ of $\CC R_n$. Similarly,  the eigenvalue of the action of $\tilde{X}_j$ on $v_{L}$, for $L\in\tau_{n}^{(1)}$ is zero, and also $\tilde{X}_{j}$ acts as zero on the trivial representation of $\CC R_n$. Thus we need to consider the actions of both $X_{i}$ and $\tilde{X}_{i}$, for all $1\leq i\leq n$, on the Gelfand--Tsetlin vectors to distinguish the nonisomorphic irreducible representations of $\CC R_n$.
\end{example}

\subsection{Kronecker product}\label{sec:Kro}
This section begins with the computation of the Kronecker product of the character of $\CC^n$ with the character of an irreducible representation of $\CC R_n$ (Theorem~\ref{prop:charc}). For this, we first define some notation. Given $K=\{i_1,\ldots,i_r\}\subseteq \{1,2,\ldots,n\}$ such that $i_{1}<\cdots<i_r$, let $\theta_{K}$ be the element of $R_{n}$ whose nonzero rows and columns are indexed by the elements of $\{1,2,\ldots,r\}$ and $K$, respectively,  such that $\theta_{K}e_{i_j}=e_{j}$ for $1\leq j\leq r$. Let $\theta^{tr}_{K}$ denote the transpose of $\theta_{K}$. Following~\cite{Solomon}, we identify $S_{r}$ with the subgroup of monoid $R_n$ consisting of matrices $\sigma=(\sigma_{i,j})_{1\leq i,j\leq n}$ such that the submatrix $(\sigma_{p,q})_{1\leq p,q \leq r}$ is a permutation matrix of order $r$ and the entries $\sigma_{s,t} = 0$ for all $r+1 \leq s, t \leq n$.

\begin{remark}\label{rm:sym}
	For $1\leq r<n$, one can also identify the symmetric group on $r$ symbols with the subgroup of monoid $R_n$ consisting of matrices 
		$\sigma=(\sigma_{i,j})_{1\leq i,j\leq n}$ such that the submatrix $(\sigma_{p,q})_{1\leq p,q \leq r}$ is a permutation matrix of order $r$ and the entries $\sigma_{s,t} = 0$ for all $r+1 \leq s, t \leq n$ except the entry $\sigma_{r+1,r+1}$ which is $1$. We denote this copy of the symmetric group on $r$ symbols by $\tilde{S}_{r}$. Given $\tilde{\tau} \in \tilde{S}_{r}$, let $\tau$ denote the element in $S_r$ obtained from $\tilde{\tau}$ by making the $(r+1,r+1)$-th entry zero. Then the map $\tilde{S}_{r}\to S_{r}$, which sends $\tilde{\tau}$ to $\tau$, is a group isomorphism. For a partition $\lambda$ of $r$, suppose that $\chi_{\lambda}$ is the character of the irreducible representation of $S_{r}$ corresponding to $\lambda$. Then the
		 pullback of $\chi_{\lambda}$ is the character $\tilde{\chi}_{\lambda}$ of the irreducible representation of 
		 $\tilde{S}_{r}$ corresponding to $\lambda$. Also for $\tilde{\tau}\in \tilde{S}_{r}$, we have $\tilde{\chi}_\lambda(\tilde{\tau})=\chi_{\lambda}(\tau)$.
\end{remark}

For $\sigma\in R_{n}$, let $I(\sigma)$ be the  set of indices of nonzero rows of $\sigma$. Let $\{e_1,\ldots,e_n\}$ be the standard basis of $\CC^n$. For a subset $K$ of $\{1,2,\ldots,n\}$, the meaning of $\sigma K=K$ is that for every $i\in K$, there exists $j\in K$ such that
$\sigma e_i=e_j$. For a positive integer $r$, define
\begin{equation*}
C_{\sigma,r}=\{K\subseteq I(\sigma) \mid |K|=r,\, \sigma K=K \}, 
\end{equation*}
where $|K|$ denotes the cardinality of $K$. For $K\in C_{\sigma,r}$, note that $\theta_{K}\sigma\theta_{K}^{tr}\in S_r$. For the character $\chi$ of a representation of $\CC S_r$ and $\sigma\in R_{n}$, define
\begin{equation}\label{eq:char_rook}
\chi^{*}(\sigma):=\sum_{K\in C_{{\sigma},r}}\chi(\theta_{K}\sigma\theta_{K}^{tr}).
\end{equation}
Then from \cite[Theorem 2.30]{Solomon} $\chi^*$ is the character of a representation of $\mathbb{C}R_{n}$. Let $|\lambda|$ denote the total number of boxes in $\lambda\in\Lambda_{\leq n}$.  For the character $\chi_\lambda$ of the irreducible representation of the symmetric group $S_{|\lambda|}$ corresponding to $\lambda\in\Lambda_{\leq n}$, the character of the irreducible representation of $\CC R_n$ corresponding to $\lambda$ is $\chi_{\lambda}^*$, obtained using~\eqref{eq:char_rook}. By combining ~\cite[Theorem 2.24]{Solomon} and \cite[Theorem 2.30]{Solomon}, we note that the characters of the irreducible representations of $\mathbb{C}R_{n}$ arise as defined in~\eqref{eq:char_rook}.

Given $\lambda\in\Lambda_{\leq n}$, we define $\lmp$ to be the set of all Young diagrams $\mu$ obtained from $\lambda$ in the following way:\label{page:lambda}
		 first remove a box from an inner corner of $\lambda$ to obtain a Young diagram $\omega$; and then add a box to an outer corner of $\omega$ to obtain $\mu$. Also, we define $\lpn$ to be the set of all Young diagrams in $\Lambda_{\leq n}$ obtained from $\lambda$ by adding a box to an outer corner.

\begin{theorem}\label{prop:charc}
Let $\lambda\in\Lambda_{\leq n}$ be a partition of $r$. The Kronecker product  of $\chi_{(1)}^*$ and $\chi_{\lambda}^*$ is given by 
\begin{displaymath}
 \chi_{(1)}^*\chi_{\lambda}^*=\sum _{\mu\in \lmp\cup \lpn}\chi_{\mu}^*.
\end{displaymath}
\end{theorem}

\begin{proof}
 For $r=0$, since $\chi_{\lambda}^*$ is the character of the trivial representation of $\CC R_n$, therefore Theorem~\ref{prop:charc} holds. Assume $1\leq r\leq n$. Let $\sigma\in R_{n}$ and $\nu=(1)$. Then using~\eqref{eq:char_rook} we have
 \begin{align}
  (\chi_{\nu}^*\chi_{\lambda}^*)(\sigma)&= \chi_{\nu}^*(\sigma)\chi_{\lambda}^*(\sigma)
  =\bigg(\underset{L\in C_{\sigma,1}}{\sum}\chi_{\nu}(\theta_{L}\sigma\theta_{L}^{tr})\bigg)\bigg(\underset{K\in C_{\sigma,r}}{\sum}\chi_{\lambda}(\theta_{K}\sigma\theta_{K}^{tr})\bigg)\nonumber\\
  &= \underset{K\in C_{\sigma,r}}{\sum}\bigg(\underset{L\in C_{\sigma,1}}{\sum}\chi_{\nu}(\theta_{L}\sigma\theta_{L}^{tr})\bigg)\chi_{\lambda}(\theta_{K}\sigma\theta_{K}^{tr})\nonumber\\
  &=\underset{K\in C_{\sigma,r}}{\sum}\bigg(\underset{\underset{L\subseteq K}{L\in C_{\sigma,1}}}{\sum}\chi_{\nu}(\theta_{L}\sigma\theta_{L}^{tr})\bigg)\chi_{\lambda}(\theta_{K}\sigma\theta_{K}^{tr})+
   \underset{K\in C_{\sigma,r}}{\sum}\bigg(\underset{\underset{L\nsubseteq K}{L\in C_{\sigma,1}}}{\sum}\chi_{\nu}(\theta_{L}\sigma\theta_{L}^{tr})\bigg)\chi_{\lambda}(\theta_{K}\sigma\theta_{K}^{tr})\label{al:exp}.
 \end{align}
 Since $\nu=(1)$, therefore, for $L\in C_{\sigma,1}$,  $\chi_{\nu}(\theta_{L}\sigma\theta_{L}^{tr})=1$. This implies that $\underset{\underset{L\subseteq K}{L\in C_{\sigma,1}}}{\sum}\chi_{\nu}(\theta_{L}\sigma\theta_{L}^{tr})$ is
 equal to the number of fixed points of $\theta_{K}\sigma\theta_{K}^{tr}$, which is also the character value at $\theta_{K}\sigma\theta_{K}^{tr}$ of the character of the defining representation $\CC^r$ of $S_r$. By the tensor identity~\cite[Equation(3.18)]{HR05} followed by the induction and restriction rules for the symmetric groups, we have
 \begin{displaymath}
  (\text{the number of fixed points of } \theta_{K}\sigma\theta_{K}^{tr}) \chi_{\lambda}(\theta_{K}\sigma\theta_{K}^{tr})= \sum_{\mu\in \lambda^{-+}}\chi_{\mu}(\theta_{K}\sigma\theta_{K}^{tr}).
 \end{displaymath}
 Also, the second term in \eqref{al:exp}, for $r=n$, is zero and otherwise it becomes
  $\underset{K\in C_{\sigma,r}}{\sum}\bigg(\underset{\underset{L\nsubseteq K}{L\in C_{\sigma,1}}}{\sum}\chi_{\lambda}(\theta_{K}\sigma\theta_{K}^{tr})\bigg)$. In the following we prove that $(\text{Ind}_{\tilde{S}_{r}}^{S_{r+1}}(\tilde{\chi}_{\lambda}))^{*}(\sigma)=\underset{K\in C_{\sigma,r}}{\sum}\bigg(\underset{\underset{L\nsubseteq K}{L\in C_{\sigma,1}}}{\sum}\chi_{\lambda}(\theta_{K}\sigma\theta_{K}^{tr})\bigg)$.
 Using \eqref{eq:char_rook} and
 the formula for an induced character (see~\cite[Equation(3.18)]{FH}), we have
  \begin{align}\label{al:sum}
  (\text{Ind}_{\tilde{S}_{r}}^{S_{r+1}}(\tilde{\chi}_{\lambda}))^{*}(\sigma)&= \underset{M\in C_{\sigma,r+1}}{\sum}\text{Ind}_{\tilde{S}_{r}}^{S_{r+1}}(\tilde{\chi}_{\lambda})(\theta_{M}\sigma\theta_{M}^{tr})\nonumber\\
 &= \sum_{M\in C_{\sigma,r+1}} \bigg(\underset{\underset{\theta_{M}\sigma \theta_{M}^{tr}(e_i)=e_i}{i\in \{1,2,\ldots,r+1\}}}{\sum} \chi_{\lambda}
 ((i,r+1)\theta_{M}\sigma \theta_{M}^{tr}(i,r+1))\bigg).
  \end{align}
  For $i\in M$, $\theta_{M}^{tr}\theta_{M}(e_i)=e_i$ and zero otherwise; similarly, $\theta_{M}\theta_{M}^{tr}$ is the identity element of $S_{r+1}$. Thus, for $i\in \{1,2,\ldots,r+1\}$, $\theta_{M}\sigma \theta_{M}^{tr}(e_i)=e_i$ if only if 
$\sigma(\theta_{M}^{tr}(e_i))=\theta_{M}^{tr}(e_i)$. So the inner sum of \eqref{al:sum} contributes zero for $M\in C_{\sigma, r+1}$ which does not contain a fixed point of $\sigma$. Also there is one-to-one correspondence between the sets
\begin{displaymath} 
 \{(M,j)| M\in C_{\sigma, r+1} \text{ and } j\in M \text{ such that }
 \sigma(e_j)=e_j\}
 \end{displaymath}
 and $\{(K,L)\in C_{\sigma,r}\times C_{\sigma,1}|
 L\nsubseteq K\}$. Therefore, for $\sigma'=(\theta_{K\cup L}(L), r+1)$, \eqref{al:sum} simplifies to 
 $$\underset{K\in C_{\sigma,r}}{\sum}\bigg(
 \underset{\underset{L\nsubseteq K}{L\in C_{\sigma, 1}}} {\sum} \tilde{\chi}_{\lambda}(
 \sigma' \theta_{K\cup L}\sigma \theta_{K\cup L}^{tr} 
\sigma'\bigg).$$
The rank of $\sigma' \theta_{K\cup L}\sigma \theta_{K\cup L}^{tr} 
\sigma'$ is $r+1$ with $(i,j)$-th same as $(i,j)$-th entry of $(\theta_{K}\sigma\theta_{K}^{tr})$, for $1\leq i,j \leq r$,
   and $(r+1,r+1)$-th entry being $1$. By Remark~\ref{rm:sym},  
 $ \tilde{\chi}_{\lambda}(\sigma' \theta_{K\cup L}\sigma \theta_{K\cup L}^{tr} 
 \sigma')=\chi_{\lambda}(\theta_{K}\sigma \theta_{K}^{tr})$. This completes the proof. 
\end{proof}

\begin{corollary}\label{prop:brat_rook}
For $\lambda\in\Lambda_{\leq n}$, the Kronecker product of $V^{\lambda}_n$ with $\CC^n$ deomposes as follows
\begin{displaymath}
 \mathbb{C}^{n}\otimes V^{\lambda}_n\cong \bigoplus_{\mu\in \lmp\cup\lpn} V^{\mu}_n.
\end{displaymath} 
\end{corollary}

\begin{proof}
The characters of $\mathbb{C}^n$ and $V^{\lambda}_n$ are
$\chi_{(1)}^*$ and $\chi_{\lambda}^*$, respectively. Since the character of the Kronecker product of representations is the product of their characters, the result follows from Proposition~\ref{prop:charc}.
\end{proof}

 We now build an undirected graph $\widehat{R}(n)$. The set $\widehat{R}_k(n)$ of the vertices  at the level $k$ of $\widehat{R}(n)$ consists of Young diagrams $\lambda$  such that $ 1\leq |\lambda| \leq \min\{k,n\}$. For $\lambda\in \widehat{R}_k(n)$ and $\mu\in \widehat{R}_{k+1}(n)$, there is an edge between $\lambda$ and $\mu$ if and only if $\mu\in \lmp\cup \lpn$. For $n\geq 3$, Figure \ref{fig:Rn3} depicts the graph $\widehat{R}(n)$ up to level $3$.

\begin{remark}\label{rem:paths}
	Let $\mathcal{P}_{k,\lambda}$ denote the set of paths (see Convention~\ref{con} for the definition of a path), in the graph $\widehat{R}(n)$, starting from 
	the Young diagram corresponding to $(1)$ at level $1$ and ending at the Young diagram $\lambda$ at level $k$. Then from Corollary~\ref{prop:brat_rook}, the multiplicity of $V^{\lambda}_n$ in $(\CC^n)^{\otimes k}$ is the cardinality of $\mathcal{P}_{k,\lambda}$.
\end{remark}

\begin{figure}[ht]\centering
 \begin{tikzpicture}[scale=0.8]

 \node (b) at (-2,-2) {$k=1$};
 \node (c) at (-2,-4) {$k=2$};
 \node (d) at (-2,-6) {$k=3$};
 
  \node (2) at  (0,-2) {$\ytableausetup{boxsize=1em} \begin{ytableau}
                                                        \empty \\
                                                       \end{ytableau}$};
  \node (3) at  (0,-4) {$\ytableausetup{boxsize=1em} \begin{ytableau}
                                                        \empty \\
                                                       \end{ytableau}$};
  \node (4) at (2,-4) {$\ytableausetup{boxsize=1em} \begin{ytableau}
                                                        \empty & \empty \\
                                                       \end{ytableau}$};
  \node (5) at (4,-4) {$\ytableausetup{boxsize=1em} \begin{ytableau}
                                                        \empty \\
                                                        \empty \\
                                                       \end{ytableau}$};
  \node (6) at (0,-6) {$\ytableausetup{boxsize=1em} \begin{ytableau}
                                                        \empty \\
                                                       \end{ytableau}$};
  \node (7) at (2,-6) {$\ytableausetup{boxsize=1em} \begin{ytableau}
                                                        \empty & \empty \\
                                                       \end{ytableau}$};
  \node (8) at (4,-6) {$\ytableausetup{boxsize=1em} \begin{ytableau}
                                                        \empty \\
                                                        \empty \\
                                                       \end{ytableau}$};
  \node (9) at (6,-6) {$\ytableausetup{boxsize=1em} \begin{ytableau}
                                                        \empty & \empty & \empty \\
                                                       \end{ytableau}$};
  \node (10) at (8,-6) {$\ytableausetup{boxsize=1em} \begin{ytableau}
                                                        \empty & \empty \\
                                                        \empty \\
                                                       \end{ytableau}$};
  \node (11) at (11,-6) {$\ytableausetup{boxsize=1em} \begin{ytableau}
                                                        \empty \\
                                                        \empty \\
                                                        \empty \\
                                                       \end{ytableau}$};
  \draw (2) -- (3) -- (6);
  \draw (2) -- (4) -- (7);
  \draw (2) -- (5) -- (11);
  \draw (3) -- (7);
  \draw (3) -- (8);
  \draw (4) -- (8);
  \draw (4) -- (9);
  \draw (4) -- (10);
  \draw (5) -- (7);
  \draw (5) -- (8);
  \draw (5) -- (10);
  \draw (5) -- (11);
 \end{tikzpicture}
 \caption{$\widehat{R}(n)$ up to level $3$}
 \label{fig:Rn3}
 \end{figure}
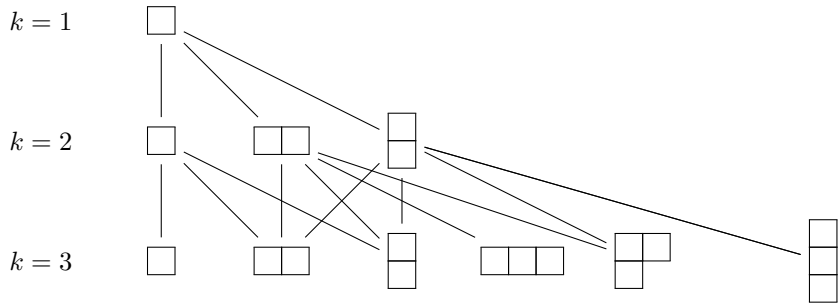
 
  Solomon~\cite[Example 3.18]{Solomon} gave a formula for the multiplicity of $V^{\lambda}_n$ in $(\CC^n)^{\otimes k}$. In Theorem~\ref{coro:bij}, we give a bijective proof of his formula, and for this we need the following definitions and notation. 

\begin{enumerate}
	\item For Young diagrams $\lambda$ and $\mu$, if $\mu$ is contained in $\lambda$, then
	we write $\mu\subset \lambda$  and in addition, if $\lambda$ and $\mu$ differ by only one box then we denote this box as $\lambda/\mu=\ytableausetup{boxsize=0.7em}
	\begin{ytableau}
	\empty \\
	\end{ytableau}$.
	\item (Maximum-entry order) For a set partition $\pi$ of $\{1,2,\ldots,k\}$, we underline the maximum entry of each block of $\pi$. Given blocks $B_1$ and $B_2$ of $\pi$, if the maximum entry of $B_1$ is strictly less than the maximum entry of $B_2$, then we write $B_1<B_2$.
	 This defines a total order on $\pi$, which is called the maximum-entry order in~\cite[Definition 3.14]{BH}. 
 For the definition of a set partition and its blocks, see Section~\ref{sec:partition}. 
	
	\item (Set-partition tableau)
	Let $\lambda$ be a Young diagram. A set tableau $T$ of shape $\lambda$ is a filling of boxes of $\lambda$ with mutually disjoint nonempty finite subsets of $\ZZ_{> 0}$.  From~\cite[Definition 3.14]{BH}, if the set of the entries of a set tableau $T$ is a set-partition of $\{1,2,\ldots,k\}$ for some $k\in\ZZ_{> 0}$, then $T$ is called a set-partition tableau. 
	  If the entries of a set tableau (respectively, set-partition tableau) strictly increase, with respect to the maximum-entry order, along the rows from left to right and along the coulmns from top to bottom, then it is called a standard set tableau (respectively, standard set-partition tableau).
	  
		\item (Robinson--Schensted--Knuth row-insertion) Given a standard set tableau $T$ and a finite nonempty subset $b$ of $\ZZ_{>0}$ which is disjoint from the entries of $T$, let 
	$(T\leftarrow b)$ denote the Robinson--Schensted--Knuth (RSK) row-insertion of $b$ into $T$. For an explicit algorithm of RSK row-insertion, we refer to~\cite[Section 7.11]{Stanley}.
\end{enumerate}

\begin{example}
	For  the set partition $\pi=\{\{1\},\{2,3\}\}$ of $\{1,2,3\}$, the maximum-entry order on the blocks of $\pi$ is depicted as $\{\underline{1}\}$, $\{2,\underline{3}\}$. For $\lambda=(2)$,
	\ytableausetup{boxsize=1.8em}
	\begin{ytableau}
		\scriptstyle\{\underline 1\} &
		\scriptstyle\{2,\underline 3\}
	\end{ytableau} 
	is the standard set-partition tableau filled with $\pi$. 
\end{example}

\begin{theorem}\label{coro:bij}
 Let $\lambda\in \Lambda_{\leq n}$ be a partition of $r$. Then the multiplicity of $V^{\lambda}_n$ in $(\CC^n)^{\otimes k}$ is $S(k,r)f^{\lambda}$, where 
 $S(k,r)$ is a Stirling number of the second kind and $f^\lambda$ is the number of standard Young tableaux of shape $\lambda$. 
\end{theorem}
\begin{proof}
	 Let  $\widehat{S}_{k,\lambda}$ denote the set of all standard set-partition tableaux of shape $\lambda$ filled with any set partition of $\{1,2,\ldots,k\}$ into $r$ blocks. The cardinality of $\widehat{S}_{k,\lambda}$ is $S(k,r)f^{\lambda}$. To prove the theorem, following Remark~\ref{rem:paths}, we give a bijection between $\mathcal{P}_{k,\lambda}$ and $\widehat{S}_{k,\lambda}$ (the idea of such a bijection comes from \cite[Section 5.3]{BHH}).

		Given $T_{k}\in\widehat{S}_{k,\lambda}$, below we define a path $(\gamma^{(1)}=(1),\gamma^{(2)},\ldots,\gamma^{(k)}=\lambda)\in \mathcal{P}_{k,\lambda}$.
		\begin{enumerate}
			\item Remove the inner corner filled with the set $b$ containing $k$ from $T_{k}$, and let $T_{k-\frac{1}{2}}$ denote the resulting standard set tableau with total number of boxes $r-1$. Let $b'=b\setminus\{k\}$. If $b'$ is empty, then $T_{k-\frac{1}{2}}$ is a standard set-partition tableau filled with a set partition of 
			$\{1,2,\ldots,k-1\}$ into $r-1$ blocks and define $T_{k-1}:=T_{k-\frac{1}{2}}$. If $b'$ is nonempty, then perform RSK row-insertion 
			$(T_{k-\frac{1}{2}}\leftarrow b')$ to obtain a standard set-partition tableau $T_{k-1}$ filled with a set partition of 
			$\{1,2,\ldots,k-1\}$ into $r$ blocks. Let $\gamma^{(k-1)}$ be the shape of $T_{k-1}$.
			\item Repeat Step $1$  until we reach $\gamma^{(1)}$ at the level $1$.
		\end{enumerate} 
	
Conversely, given a path $(\gamma^{(1)}=(1),\gamma^{(2)},\ldots,\gamma^{(k)}=\lambda)\in\mathcal{P}_{k,\lambda}$, we recursively construct a sequence $T_{1},T_{2},\ldots,T_{k}$ such that $T_{i}\in\widehat{S}_{i, \gamma^{(i)}}$ for $1\leq i\leq k$. Note that $T_{1}$ is the set-partition tableau of shape $(1)$ filled with $\{\underline{1}\}$. Note that by the definition of a path in $\mathcal{P}_{k,\lambda}$, $\gamma^{(i)}\in(\gamma^{(i-1)})^{-+}\cup (\gamma^{(i-1)})^{+,n}$, therefore,
 for $i=2,\ldots,k$, we can construct $T_{i}$ from $T_{i-1}$ by performing the following cases:
\begin{enumerate} 
	\item If $\gamma^{(i)}/ \gamma^{(i-1)}=\ytableausetup{boxsize=0.7em}
\begin{ytableau}
\empty \\
\end{ytableau}$, then $T_i$ is same as $T_{i-1}$ with the additional box of $T_{i}$ filled with $\{\underline{i}\}$. 
	\item Suppose $\gamma^{(i)}$ is obtained from $\gamma^{(i-1)}$ by first removing an inner corner, resulting into a Young diagram $\gamma^{(i-\frac{1}{2})}$, and then adding an outer corner to $\gamma^{(i-\frac{1}{2})}$. Then using RSK row-insertion algorithm, there exists a unique standard set tabelau $T_{i-\frac{1}{2}}$ of shape $\gamma^{(i-\frac{1}{2})}$ and a unique set partition $b$ of $\{1,2,\ldots,i-1\}$ such that 
	$T_{i-1}=(T_{i-\frac{1}{2}}\leftarrow b)$. Note that $\gamma^{(i-\frac{1}{2})}\subset \gamma^{(i)}$ such that 
	$\gamma^{(i)}/\gamma^{(i-\frac{1}{2})}=\ytableausetup{boxsize=0.7em}
\begin{ytableau}
\empty \\
\end{ytableau}$. The standard set-partition tableau $T_{i}$ is same as $T_{(i-\frac{1}{2})}$ with the additional box  $\gamma^{(i)}/\gamma^{(i-\frac{1}{2})}=\ytableausetup{boxsize=0.7em}
\begin{ytableau}
\empty \\
\end{ytableau}$ filled with $b\cup\{\underline{i}\}$.  \qedhere
\end{enumerate} 
\end{proof}

\begin{example}
We illustrate that the path $\bigg(\,\ytableausetup{boxsize=1.5em}
\begin{ytableau}
\empty \\
\end{ytableau},
\begin{ytableau}
\empty & \empty \\
\end{ytableau},
\begin{ytableau}
\empty \\
\empty \\
\end{ytableau}\,\bigg)$	
in $\widehat{R}(3)$ corresponds to the standard set-partition tableau
$\ytableausetup{boxsize=1.8em}
\begin{ytableau}
 \scriptstyle\{\underline 1\}\\
 \scriptstyle\{2,\underline 3\}
\end{ytableau}.$
For the given path, we have $\gamma^{(1)}=\ytableausetup{boxsize=1.5em} 
\begin{ytableau}
\empty \\
\end{ytableau},
\gamma^{(2)}=\begin{ytableau}
\empty & \empty \\
\end{ytableau},
\gamma^{(3)}=\begin{ytableau}
\empty \\
\empty \\
\end{ytableau}$. So
$T_{1}=\ytableausetup {boxsize=1.5em} 
\begin{ytableau}
 \scriptstyle\{\underline 1\}
\end{ytableau}$, and since $\gamma^{(2)}/\gamma^{(1)}= \ytableausetup {boxsize=1.5em} \begin{ytableau} 
\empty \\
\end{ytableau} $, we have 
$T_{2}= \ytableausetup {boxsize=1.5em}
\begin{ytableau}
 \scriptstyle\{\underline 1\} &
 \scriptstyle\{\underline 2\}
\end{ytableau}$. Also, $\gamma^{(3)}$ is obtained from $\gamma^{(2)}$ by first removing a box which results in $\gamma^{(\frac{5}{2})}=\ytableausetup {boxsize=1.5em} \begin{ytableau}
\empty \\
\end{ytableau}$ and then adding a box to $\gamma^{(\frac{5}{2})}$. By RSK row-insertion,  we have 
$T_{\frac{5}{2}}=\ytableausetup {boxsize=1.5em}
\begin{ytableau}
 \scriptstyle\{\underline 1\}
\end{ytableau}$ and 
$b=\{\underline{2}\}$ such that $
T_2=T_{\frac{5}{2}}\leftarrow b.$
By the algorithm discussed in the proof of Theorem~\ref{coro:bij}, $T_{3}$ is same as $T_{\frac{5}{2}}$ with the additional box $\gamma^{(3)}/\gamma^{(\frac{5}{2})}=\ytableausetup{boxsize=1.5em}
\begin{ytableau}
\empty \\
\end{ytableau}$ 
filled with $\{2,\underline{3}\}$, i.e, $T_{3}=$
$ \ytableausetup{boxsize=1.8em}
\begin{ytableau}
 \scriptstyle\{\underline 1\}\\
 \scriptstyle\{2,\underline 3\}
\end{ytableau}$.
Similarly, the paths $\bigg(\gamma^{(1)}=\ytableausetup{boxsize=1.5em} 
\begin{ytableau}
\empty \\
\end{ytableau},
\gamma^{(2)}=\begin{ytableau}
\empty \\
\empty \\
\end{ytableau},
\gamma^{(3)}=\begin{ytableau}
\empty \\
\empty \\
\end{ytableau}\,\bigg)$ and 
$\bigg(\gamma^{(1)}=\begin{ytableau}
\empty \\
\end{ytableau},
\gamma^{(2)}=\begin{ytableau}
\empty \\
\end{ytableau},
\gamma^{(3)}=\begin{ytableau}
\empty \\
\empty \\
\end{ytableau}\,\bigg)$ correspond to the standard
set-partition tableaux 
$\ytableausetup {boxsize=1.8em}
\begin{ytableau}
 \scriptstyle\{\underline 2\}\\
 \scriptstyle\{1,\underline 3\}
\end{ytableau}$ and 
$\begin{ytableau}
 \scriptstyle\{1,\underline 2\}\\
 \scriptstyle\{\underline 3\}
\end{ytableau}$, respectively.
\end{example}

\subsection{Modified induction and restriction rules}\label{sec:ind_res} 
  This section defines modified induction and modified restriction rules between $\CC R_{n-1}$ and  $\CC R_n$. We prove the corresponding Frobenius reciprocity. An important application of this is to prove, in Theorem~\ref{thm:item2}(ii), that the branching rule for inclusion $\CC I_{k+\frac{1}{2}}\subset\CC I_{k+1}$ of totally propagating partition algebras is multiplicity free.

For $\lambda\in\Lambda_{\leq n-1}$,  define $\lp$ to be the set of all Young diagrams in $\Lambda_{\leq n}$ obtained from $\lambda$ by adding a box to an outer corner. Define the modified induction rule as follows
\begin{equation*}
\widehat{\Ind}^{\CC R_n}_{\CC R_{n-1}}\big( \bigoplus_{\lambda \in \Lambda_{\leq n-1}} (V^{\lambda}_{n-1})^{\oplus m_{\lambda}}\big)= \bigoplus_{\lambda \in \Lambda_{\leq n-1}} \big(\bigoplus_{\nu\in\lambda^+} V^{\nu}_{n}\big)^{\oplus m_{\lambda}},\quad \text{ where } m_{\lambda}\in\ZZ_{> 0}.
\end{equation*}

\begin{proposition}[Tensor identity]\label{prop:tensor_identity} For a representation $M$ of $\CC R_n$, we have 
	\begin{equation}\label{eq:tensor_identity}
	\widehat{\Ind}^{\CC R_n}_{\CC R_{n-1}}(\Res^{\CC R_n}_{\CC R_{n-1}}M)\cong (M\otimes \CC^n).
	\end{equation}
\end{proposition}

\begin{proof}
	When $M=V^{\lambda}_n$ for $\lambda\in\Lambda_{\leq n}$, \eqref{eq:tensor_identity} holds using Proposition~\ref{prop:res}, the definition of 
	$\widehat{\Ind}^{\CC R_n}_{\CC R_{n-1}}(-)$ and Corollary~\ref{prop:brat_rook}. For an arbitrary representation of $\CC R_n$, the result follows by using the complete reducibility and by observing that $\widehat{\Ind}^{\CC R_n}_{\CC R_{n-1}}(-)$ preserves the direct sums.
\end{proof}

\begin{proposition}\label{prop:ind_res_tensor} For $k\in\ZZ_{\geq 0}$, we have
	\begin{equation}\label{eq:tensor_ind_res}
	(\widehat{\Ind}^{\CC R_n}_{\CC R_{n-1}}({\Res}_{\mathbb{C}R_{n-1}}^{\mathbb{C}R_{n}}V^{\emptyset}_n))^{k}\cong(\mathbb{C}^{n})^{\otimes k}.
	\end{equation}
\end{proposition}

\begin{proof}
	For $k=0$, both sides of~\eqref{eq:tensor_ind_res} are isomorphic to the trivial representation $\CC$ of $\CC R_n$.
We prove the result using induction on $k$. For $k=1$,
 $\widehat{\Ind}^{\CC R_n}_{\CC R_{n-1}}({\Res}_{\mathbb{C}R_{n-1}}^{\mathbb{C}R_{n}}V^{\emptyset}_{n})\cong V^{(1)}_{n}\cong\CC^n$. For $k\geq 2$, suppose that~\eqref{eq:tensor_ind_res} holds for some $k-1$. Then, 
 \begin{displaymath}
 (\widehat{\Ind}^{\CC R_n}_{\CC R_{n-1}}({\Res}_{\mathbb{C}R_{n-1}}^{\mathbb{C}R_{n}}V^{\emptyset}_{n}))^{k}=
\widehat{\Ind}^{\CC R_n}_{\CC R_{n-1}}({\Res}_{\mathbb{C}R_{n-1}}^{\mathbb{C}R_{n}}((\widehat{\Ind}({\Res}_{\mathbb{C}R_{n-1}}^{\mathbb{C}R_{n}}V^{\emptyset}_{n}))^{k-1}).
\end{displaymath}
Now using the induction hypothesis for $k-1$ and Proposition~\ref{prop:tensor_identity}, we get the required result.
\end{proof}

 For $\lambda\in\Lambda_{\leq n}$,  define $\lm$ to be the set of all Young diagrams in $\Lambda_{\leq n-1}$ obtained from $\lambda$ by removing a box from an inner corner. Define the modified restriction rule as follows
 \begin{equation*}
\widehat{\Res}^{\CC R_n}_{\CC R_{n-1}}\big(\bigoplus_{\lambda\in\Lambda_{\leq n} } (V^{\lambda}_{n})^{\oplus n_{\lambda}}\big)=\bigoplus_{\lambda\in\Lambda_{\leq n}}\big(\bigoplus_{\mu\in\lambda^{-}}V^\mu_{n-1}\big)^{\oplus n_{\lambda}}, \quad \text{ where } n_{\lambda}\in\ZZ_{> 0}.
\end{equation*}

\begin{proposition}[Frobenius reciprocity]\label{prop:induct_restrict}
	Let $V$ and $W$ be representations of $\CC R_n$ and $\CC R_{n-1}$, respectively. Then
	\begin{equation}\label{eq:fr}
	\Hom_{\CC R_n}(\widehat{\Ind}^{\CC R_n}_{\CC R_{n-1}}(W),V)\cong \Hom_{\CC R_{n-1}}(W,\widehat{\Res}^{\CC R_n}_{\CC R_{n-1}}(V)).
	\end{equation}
\end{proposition}
\begin{proof}
Since the rook monoid algebra over $\mathbb{C}$ is semisimple, it is enough to prove \eqref{eq:fr} for $V=V^{\lambda}_{n}$ and $W=V^{\mu}_{n-1}$, where $\lambda\in \Lambda_{\leq n}$ and $\mu\in \Lambda_{\leq n-1}$, respectively. Then,
	\begin{align*}
		\Hom_{\CC R_n}(\widehat{\Ind}^{\CC R_n}_{\CC R_{n-1}}(V^\mu_{n-1}),V^\lambda_{n})&=\Hom_{\CC R_n}(\bigoplus_{\nu\in\mu^{+}}V^{\nu}_{n},V^{\lambda}_n)
		= \begin{cases}
		\Hom_{\CC R_n}(V^\lambda_n,V^\lambda_n) & \text{if } \lambda\in\mu^{+},
	\\	\{0\} & \text{otherwise}.
		\end{cases}
	\end{align*}
	Similarly, 
	\begin{align*}
	\Hom_{\CC R_{n-1}}(V^\mu_{n-1},\widehat{\Res}^{\CC R_n}_{\CC R_{n-1}} (V^\lambda_n))&=\Hom_{\CC R_{n-1}}(V^{\mu}_{n-1},\bigoplus_{\nu'\in\lambda^{-}}V^{\nu'}_{n-1})\\
	&= \begin{cases}
	\Hom_{\CC R_{n-1}}(V^\mu_{n-1},V^\mu_{n-1}) &  \text{if } \mu\in\lambda^{-},
	\\	\{0\} & \text{otherwise}.
	\end{cases}
	\end{align*}
	Since $\lambda\in \mu^{+}$ if and only if $\mu\in\lambda^{-}$, and  $\Hom_{\CC R_n}(V^\lambda_{n},V^\lambda_n)\cong \CC\cong\Hom_{\CC R_{n-1}}(V^{\mu}_{n-1},V^{\mu}_{n-1})$, therefore \eqref{eq:fr} holds.
\end{proof}

\section{Totally propagating partition algebras}\label{sec:tppa_rep}

The totally propagating partition algebra (Definition~\ref{def:total_par}) is a subalgebra of the partition algebra (Definition~\ref{def:par_alg}).   In this section, we give an indexing set of the irreducible representations of totally propagating partition algebras and show that the Bratteli diagram for the tower of these algebras is a simple graph.  Let us first begin with a brief overview of partition algebras.

\subsection{Partition algebras}\label{sec:partition} 
 A set partition of  a finite set $\mathcal{A}$ is a collection $\{B_1,B_2,\ldots,B_s\}$ of mutually disjoint nonempty sets such that $\sqcup_{p=1}^{s}B_p=\mathcal{A}$, where $s\in\ZZ_{> 0}$. The sets $B_p$'s are called the blocks of the given set partition.

Given a set partition of $\{1,2,\ldots,k,1',2',\ldots,k'\}$, draw an undirected graph whose vertices are arranged in two rows such that the top row consists of the vertices $1,2,\ldots,k$, the bottom row consists of the vertices $1',2',\ldots,k'$; and there is a path between two vertices if and only if both vertices lie in the same block of the set partition. This graph is called the partition diagram corresponding to the given set partition. The connected components of a partition diagram are called its blocks, and these correspond to the blocks of the associated set partition.

  \begin{example}
  	The set partition $\{\{1,2,1',3'\},\{4,2'\},\{3,4'\}\}$ of $\{1,2,3,4,1',2',3',4'\}$ corresponds to the partition diagram in Figure \ref{fig:pd}:
  	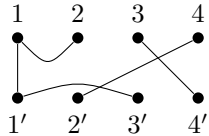
\begin{figure}[ht]\centering
  	\begin{tikzpicture}
  	[scale=0.8,
  	mycirc/.style={circle,fill=black, minimum size=0.1mm, inner sep = 1.5pt}]
  	
  	\node[mycirc,label=above:{$1$}] (n1) at (0,1) {};
  	\node[mycirc,label=above:{$2$}] (n2) at (1,1) {};
  	\node[mycirc,label=above:{$3$}] (n3) at (2,1) {};
  	\node[mycirc,label=above:{$4$}] (n4) at (3,1) {};
  	\node[mycirc,label=below:{$1'$}] (n1') at (0,0) {};
  	\node[mycirc,label=below:{$2'$}] (n2') at (1,0) {};
  	\node[mycirc,label=below:{$3'$}] (n3') at (2,0) {}; 
  	\node[mycirc,label=below:{$4'$}] (n4') at (3,0) {}; 
  
  	\draw (n1)--(n1');
  	\draw (n1)..controls(0.5,0.5).. (n2);
  	\draw (n1')..controls(1,0.3)..(n3');
  	\draw (n4)--(n2');
  \draw (n3)--(n4');
  	\end{tikzpicture}
        \caption{An example of partition diagram}
        \label{fig:pd}
  	\end{figure}
\end{example} 

Let $A_k$ denote the set of all partition diagrams on $\{1,2,\ldots,k,1',2'\ldots,k'\}$. For $d_1,d_2\in A_k$, the composition $d_1\circ d_2\in A_{k}$ is the concatenation of partition diagrams obtained by placing $d_1$ above $d_2$, identifying the bottom row of $d_1$ with the top row of $d_2$, and excluding the connected components that lie entirely in the middle row. With respect to the composition of partition diagrams, $A_k$ is a monoid. Let $A_{k+\frac{1}{2}}$ be the submonoid of $A_{k+1}$ consisting of partition diagrams whose vectices $(k+1)$
and $(k+1)'$ are always in the same block.

\begin{definition}\phantomsection\label{def:par_alg}
	\begin{enumerate}[(i)]
		\item For $\xi\in \CC$, let $\CC A_{k}(\xi)=\CC\text{-span}\{d\mid d\in A_k\}$.
Given basis elements $d_1,d_2\in A_k$, define a multiplication in $\CC A_{k}(\xi)$ as follows:
\begin{equation}\label{eq:mul_rule}
d_1d_2:=\xi^l d_1\circ d_2,
\end{equation}
where $l$ is the number of connected components that lie entirely in the middle row while computing $d_1\circ d_2$. With respect to the multiplication~\eqref{eq:mul_rule}, $\CC A_k(\xi)$ is a unital associative algebra over $\CC$. 

\item	Let
	$\CC A_{k+\frac{1}{2}}(\xi)=\CC\text{-span}\{d\mid d\in A_{k+\frac{1}{2}}\}$.
	The subspace $\CC A_{k+\frac{1}{2}}(\xi)$ is a subalgebra of $\CC A_{k+1}(\xi)$. 
	\end{enumerate}
Both $\CC A_{k}(\xi)$ and $\CC A_{k+\frac{1}{2}}(\xi)$ are called partition algebras.
\end{definition}

\begin{example}
	Let $d_1$ and $d_2$ be the partition diagrams corresponding to $\{\{1,3\},\{2,1'\},\{4\},\{2',3'\},\{4'\}\}$ and $\{\{1,4'\}, \{2\},\{3\},\{4\},\{1'\},\{2',3'\}\}$, respectively. The multiplication $d_1d_2$ in $\CC A_4(\xi)$ is illustrated in Figure \ref{fig:mult}. 
\begin{figure}[ht]\centering
		
	\begin{tikzpicture}
	[scale=0.7,
	mycirc/.style={circle,fill=black, minimum size=0.1mm, inner sep = 1.5pt}]
	\node (1) at (-1,0.5) {$d_1 =$};
	\node[mycirc,label=above:{$1$}] (n1) at (0,1) {};
	\node[mycirc,label=above:{$2$}] (n2) at (1,1) {};
	\node[mycirc,label=above:{$3$}] (n3) at (2,1) {};
	\node[mycirc,label=above:{$4$}] (n4) at (3,1) {};
	
	\node[mycirc,label=below:{$1'$}] (n1') at (0,0) {};
	\node[mycirc,label=below:{$2'$}] (n2') at (1,0) {};
	\node[mycirc,label=below:{$3'$}] (n3') at (2,0) {}; 
	\node[mycirc,label=below:{$4'$}] (n4') at (3,0) {};

	\draw (n2)--(n1');
	\draw (n1)..controls(1,0.5).. (n3);

	\draw (n2')..controls(1.5,0.5)..(n3');
	
	\node (1) at (-1,-2) {$d_2 =$};
	\node[mycirc,label=above:{$1$}] (n6) at (0,-1.5) {};
	\node[mycirc,label=above:{$2$}] (n7) at (1,-1.5) {};
	\node[mycirc,label=above:{$3$}] (n8) at (2,-1.5) {};
	\node[mycirc,label=above:{$4$}] (n9) at (3,-1.5) {};

	\node[mycirc,label=below:{$1'$}] (n6') at (0,-2.5) {};
	\node[mycirc,label=below:{$2'$}] (n7') at (1,-2.5) {};
	\node[mycirc,label=below:{$3'$}] (n8') at (2,-2.5) {}; 
	\node[mycirc,label=below:{$4'$}] (n9') at (3,-2.5) {};

	\draw (n6)--(n9');
	\draw (n7')..controls(1.5,-2)..(n8');

	\draw[dashed] (n1')..controls(-0.5,-1)..(n6);
	\draw[dashed] (n2')..controls(0.5,-1)..(n7);
	\draw[dashed] (n3')..controls(1.5,-1)..(n8);
	\draw[dashed] (n4')..controls(2.5,-1)..(n9);

\node (1) at (-1.2,-4.5) {$d_1d_2 =\xi^{2}$};
\node[mycirc,label=above:{$1$}] (n1) at (0,-4) {};
\node[mycirc,label=above:{$2$}] (n2) at (1,-4) {};
\node[mycirc,label=above:{$3$}] (n3) at (2,-4) {};
\node[mycirc,label=above:{$4$}] (n4) at (3,-4) {};

\node[mycirc,label=below:{$1'$}] (n1') at (0,-5) {};
\node[mycirc,label=below:{$2'$}] (n2') at (1,-5) {};
\node[mycirc,label=below:{$3'$}] (n3') at (2,-5) {}; 
\node[mycirc,label=below:{$4'$}] (n4') at (3,-5) {};

\draw (n2)--(n4');
\draw (n1)..controls(1,-4.5).. (n3);
\draw (n2')..controls(1.5,-4.5)..(n3');

\end{tikzpicture}
\caption{Multiplication}
\label{fig:mult}
\end{figure}
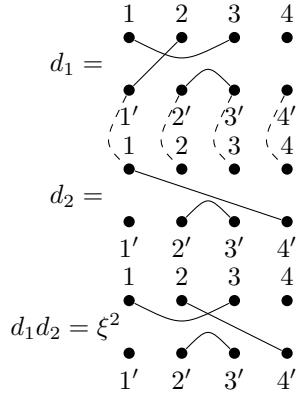
\end{example}

We recall from~\cite[p. 5]{Jones} the orbit basis of $\CC A_k(\xi)$ which will be relevant in Section~\ref{sec:JM}. For $d_1,d_2\in A_k$, we say $d_1$ is coarser than $d_2$ if for $i$ and $j$ in the same block of $d_2$ then $i$ and $j$ are in the same block of $d_1$. The notation $d_1\leq d_2$ means that $d_1$ is coarser than $d_2$ and this defines a partial order on $A_k$. 
 For $d\in A_k$, define $x_d\in \CC A_k(\xi)$ to be the element uniquely satisfying the following relation
\begin{equation}\label{def:x_d}
d=\sum_{d'\leq d} x_{d'}.
\end{equation}
By linearly extending the partial order, it can be seen that the transition matrix between $\{d\mid d\in A_k\}$ and $\{x_d\mid d\in A_k\}$ is unitriangular. So $\{x_d\mid d\in A_k\}$ is  also a basis, called the orbit basis, of 
$\CC A_k(\xi)$. The structure constants of $\CC A_k(\xi)$ with respect to the   
orbit basis were first given in online notes~\cite{Ram2010S} and later these also appeared in
~\cite[Theorem 4.8]{BH19} and~\cite[Lemma 3.1]{MS}.

For $d_1,d_2\in A_k$, a block in $d_1\circ d_2$ is called an internal block if it lies entirely in the middle row while computing $d_1\circ d_2$.
For $1\leq i,j\leq k$, whenever $i'$ and $j'$ are in the same block in $d_1$ if and only if $i$ and $j$ are in the same block in $d_2$, then we say that
the bottom row of $d_1$ matches with the top row of $d_2$. 

\begin{lemma}\label{lm:obprod}
	For $d_1, d_2 \in A_k$, the multiplication of $x_{d_1}$ and $x_{d_2}$ in $\CC A_k(\xi)$ is given by 
	\begin{align*}
	x_{d_1}x_{d_2} = 
	\begin{cases}
	\sum\limits_{d} c_d x_d & \mbox{if the bottom row of } d_1 \mbox{ matches with the top row of } d_2, \\
	0 & \mbox{otherwise,}
	\end{cases}
	\end{align*}
	where the sum is taken over all those $d$ in $A_k$ such that $d$ is coarser than $d_1 \circ d_2$ and the coarsening is done by connecting a block of $d_1$ which is contained entirely in the top row of $d_1$ with a block of $d_2$ which is contained entirely in the bottom row of $d_2$ and 
	\begin{equation*}
	c_d = (\xi - |d|)_{[d_1 \circ d_2]}, 
	\end{equation*}
	where $|d|$ is the number of blocks in $d$, $[d_1 \circ d_2]$ is the number of internal blocks in $d_1 \circ d_2$, and for any polynomial $f(\xi)$ in $\xi$, $b \in \ZZ_{\geq 0},$
	\begin{align*}
	(f(\xi))_b := 
	\begin{cases}
	f(\xi)(f(\xi)-1) \cdots (f(\xi)-b+1) & \mbox{if }  b > 0,\\
	1 & \mbox{if } b = 0.
	\end{cases}
	\end{align*}
	\end{lemma}

{\bf Totally propagating partition algebras}.
A block $B$ of a partition diagram in $A_k$ is called a propagating block if $B$ contains vertices from both top and bottom rows, i.e., 
\begin{displaymath}
B\cap \{1,2,\ldots,k\}\neq \emptyset \quad \text{and}\quad B\cap \{1',2',\ldots,k'\}\neq \emptyset.
\end{displaymath}
Let $I_{k}$ be the submonoid of $A_k$ consisting of partition diagrams each of whose blocks are propagating blocks. Let $I_{k+\frac{1}{2}}:=A_{k+\frac{1}{2}}\cap I_{k+1}$.

\begin{definition}\phantomsection\label{def:total_par}
	For $t\in\frac{1}{2}\ZZ_{> 0}$, define $\CC I_{t}=\CC\text{-span}\{d\mid d\in I_t\}$. The subspace $\CC I_t$ is a subalgebra of $\CC A_t(\xi)$.
	We call $\CC I_t$ the totally propagating partition algebra.
\end{definition}

 An important observation is that the multiplication in $\CC I_t$  does not depend on the multiplication factor $\xi$. Specifically, for $d_1,d_2\in I_t$, the multiplication in $\CC I_t$ is given by
\begin{displaymath}
d_1d_2=d_1\circ d_2, 
\end{displaymath} 
since there are no connected components which lie entirely in the middle row while computing 
$d_1\circ d_2$.

For $d\in I_t$, any partition diagram coarser than $d$ is also an element of $I_t$ and therefore, $x_d\in \CC I_t$. Lemma~\ref{lm:obprod} implies the following corollary which specializes $x_{d_1}x_{d_2}$ for $d_1,d_2\in I_t$.

\begin{corollary}\label{coro:obprod}
	For $d_1, d_2 \in I_t$, the multiplication of $x_{d_1}$ and $x_{d_2}$ in $\CC I_t$ is given by 
	\begin{align*}
	x_{d_1}x_{d_2} = 
	\begin{cases}
	 x_{d_1\circ d_2} & \mbox{if the bottom row of } d_1 \mbox{ matches with the top row of } d_2, \\
	0 & \mbox{otherwise}.
	\end{cases}
	\end{align*}
\end{corollary}
 
\subsection{Schur--Weyl dualities}\label{sec:SWD}

The symmetric group $S_n$ acts on the vector space $V=\CC^n$ with the standard basis $\{e_1,e_2,\ldots,e_n\}$ naturally, i.e.,
$\beta.e_{i}=e_{\beta(i)}$, where $\beta\in S_n$ and $1\leq i\leq n$. The $k$-fold tensor space $V^{\otimes k}$ with a basis
\begin{equation}\label{eq:tensor}
\{e_{i_1}\otimes e_{i_2}\otimes \cdots \otimes e_{i_k}\mid 1\leq i_1,i_2,\ldots,i_k\leq n\}
\end{equation}
is a representation of $S_n$ with respect to the diagonal action.  
We identify $S_{n-1}$ with the subgroup consisting of permutation matrices in $S_n$ which fix $e_n$. With this identification, $V^{\otimes (k+\frac{1}{2})}:=V^{\otimes k}\otimes e_n$ is a representation of $S_{n-1}$.

The partition algebra $\CC A_k(n)$ acts on $V^{\otimes k}$ on the right by the following map:
 \begin{align*}
 \phi_{k}:\CC A_k(n)\to \End_{\CC}(V^{\otimes k}),
 \end{align*}
 where, for $d\in A_k$, the map $\phi_k(d)$ on the basis element $e_{i_1}\otimes e_{i_{2}}\otimes\cdots\otimes e_{i_k}$ is given by
\begin{equation*}
(e_{i_1}\otimes e_{i_{2}}\otimes\cdots\otimes e_{i_k})\phi_k(d)=\sum_{1\leq i_{1'},i_{2'},\ldots,i_{k'}\leq n} (\phi_k(d))^{i_1,i_2\ldots,i_k}_{i_{1'},i_{2'}\ldots,i_{k'}}e_{i_{1'}}\otimes e_{i_{2'}}\otimes\cdots\otimes e_{i_{k'}}
\end{equation*}
with the coefficients
\begin{align}\label{al:coeff}
(\phi_k(d))^{i_1,i_2\ldots,i_k}_{i_{1'},i_{2'}\ldots,i_{k'}}=\begin{cases}
1 & \text{ if } i_r=i_s \text{ when $r$ and $s$ are in the same block of }d,\\
0& \text{ otherwise}.
\end{cases}
\end{align}

By the definition of $x_d$ in \eqref{def:x_d}, we see that the coefficient of $e_{i_{1'}}\otimes e_{i_{2'}}\otimes\cdots\otimes e_{i_{k'}}$ in the linear expansion of 
$(e_{i_1}\otimes e_{i_{2}}\otimes\cdots\otimes e_{i_k})\phi(x_d)$ with respect to the basis~\eqref{eq:tensor} is 
\begin{align}\label{al:orbit}
(\phi_k(x_d))^{i_1,i_2\ldots,i_k}_{i_{1'},i_{2'}\ldots,i_{k'}}=\begin{cases}
1 & \text{ if } i_r=i_s \text{ if and only if $r$ and $s$ are in the same block of }d,\\
0& \text{ otherwise}.
\end{cases}
\end{align}

Let $\phi_{k+\frac{1}{2}}:\CC A_{k+\frac{1}{2}}(n)\to\End_{\CC}(V^{\otimes{(k+\frac{1}{2}})})$ denote the restriction $\phi_{(k+1)\mid_{\CC A_{k+\frac{1}{2}}(n)}}$ of the map $\phi_{k+1}$ to $\CC A_{k+\frac{1}{2}}(n)$. The map $\phi_{k+\frac{1}{2}}$ gives a right action of $\CC A_{k+\frac{1}{2}}(n)$ on $V^{\otimes{(k+\frac{1}{2}})}$. We recall Schur--Weyl dualities between partition algebras and symmetric groups from~\cite[Theorem 3.6]{HR05}.

 \begin{theorem}\phantomsection\label{thm:swd_partition}
 	\begin{enumerate}[(i)]
 	\item The image of $\phi_k:\CC A_k(n)\to \End_{\CC}(V^{\otimes k})$ is $\End_{\CC S_n}(V^{\otimes k})$. The kernel of $\phi_k$ is $\CC$-span $\{x_d\mid d \text{ has more than $n$ } blocks\}$. Thus, for $n\geq 2k$, $\CC A_k(n)\cong \End_{\CC S_n}(V^{\otimes k})$.
 	\item The image of $\phi_{k+\frac{1}{2}}:\CC A_{k+\frac{1}{2}}(n)\to \End_{\CC}(V^{\otimes (k+\frac{1}{2})})$ is $\End_{\CC S_{n-1}}(V^{\otimes (k+\frac{1}{2})})$. The kernel of $\phi_{k+\frac{1}{2}}$ is $\CC$-span $\{x_d\mid d \text{ has more than $n$ } blocks\}$. Thus, for $n\geq 2k+1$, $\CC A_{k+\frac{1}{2}}(n)\cong \End_{\CC S_{n-1}}(V^{\otimes (k+\frac{1}{2})})$.
 	\end{enumerate}
 \end{theorem}

Let $\tilde{\psi}_k:\CC R_n\to \End_{\CC}(V^{\otimes k})$ be the algebra homomorphism arising from the action of $\CC R_n$ on $V^{\otimes k}$ and let $\tilde{\phi}_k$ be the restriction map
$\phi_{k\mid_{\CC I_k}}$ of $\phi_k$ to $\CC I_k$. For $d\in I_k$, since each block of $d$ is propagating, therefore from \eqref{al:coeff} we have:
\begin{align*}
(e_{i_1}\otimes e_{i_{2}}\otimes\cdots\otimes e_{i_k})\tilde{\phi}_{k}(d)=\begin{cases}
e_{i_{1'}}\otimes e_{i_{2'}}\otimes \cdots\otimes e_{i_{k'}}& \text{if $i_r=i_s$ when
 $r$ and $s$ are in}\\ &\text{the same block of $d$},\\
0 &\text{ otherwise}.
\end{cases}
\end{align*}

 The following theorem is the Schur--Weyl duality~\cite[Theorem 1]{KM08} between the actions of $\CC I_k$ and
$\CC R_n$ on $V^{\otimes k}$.
\begin{theorem}\phantomsection
	\label{thm:swd_rook}
	\begin{enumerate}[(i)]
		\item The image of $\tilde{\phi}_{k}:\CC I_k\to \End_{\CC}(V^{\otimes k})$
		is $\End_{\CC R_n}(V^{\otimes k})$. The kernel of $\tilde{\phi}_{k}$ is $\CC$-span 
		$\{x_d\mid d\in I_k \text{ and } d \text{ has more than } n \text{ blocks}\}$.
		 In particular, when 
		 $n\geq k$, $\CC I_k\cong \End_{\CC R_n}(V^{\otimes k})$.
		\item The image of $\tilde{\psi}_{k}:\CC R_n\to \End_{\CC}(V^{\otimes k})$ is  $
		\End_{\CC I_k}(V^{\otimes k})$.
	\end{enumerate}
\end{theorem}

Let $\tilde{\psi}_{k+\frac{1}{2}}:\CC R_{n-1}\to \End_{\CC}(V^{\otimes (k+\frac{1}{2})})$ be the algebra homomorphism arising from the action of $\CC R_{n-1}$ on $V^{\otimes (k+\frac{1}{2})}$ 
and let $\tilde{\phi}_{k+\frac{1}{2}}$ be the restriction map
$\phi_{k+\frac{1}{2}\mid_{\CC I_{k+\frac{1}{2}}}}$ of $\phi_{k+\frac{1}{2}}$ to $\CC I_{k+\frac{1}{2}}$. Then as a corollary of Theorem~\ref{thm:swd_partition}
and Theorem~\ref{thm:swd_rook}, we obtain:

\begin{corollary}\phantomsection
	\label{thm:swdfrook_half}
		\begin{enumerate}[(i)]
			\item The image of $\tilde{\phi}_{k+\frac{1}{2}}:\CC I_{k+\frac{1}{2}}\to \End_{\CC}(V^{\otimes (k+\frac{1}{2})})$
			is $\End_{\CC R_{n-1}}(V^{\otimes (k+\frac{1}{2})})$. The kernel of $\tilde{\phi}_{k+\frac{1}{2}}$ is 
		$\CC\text{-span } \{x_d\mid d\in I_{k+\frac{1}{2}} \text{ and } d \text{ has more than } n \text{ blocks}\}$. In particular, when 
			$n\geq k+1$, $\CC I_{k+\frac{1}{2}}\cong \End_{\CC R_{n-1}}(V^{\otimes (k+\frac{1}{2})})$.
		
			\item  The image of $\tilde{\psi}_{k+\frac{1}{2}}:\CC R_{n-1}\to \End_{\CC}(V^{\otimes (k+\frac{1}{2})})$ is 
			$\End_{\CC I_{k+\frac{1}{2}}}(V^{\otimes (k+\frac{1}{2})})$.
		\end{enumerate}
	
\end{corollary}

\subsection{Irreducible representations of totally propagating partition algebras}\label{sec:irrep_tppa}
We describe indexing sets of the irreducible representations of $\CC I_k$ and $\CC I_{k+\frac{1}{2}}$, and determine the branching rule for $\CC I_k\subset\CC I_{k+\frac{1}{2}}\subset \CC I_{k+1}$. It is important to note that the embedding~\eqref{eq:embedding} $\CC I_{k}\subset \CC I_{k+\frac{1}{2}}$ is not induced from the embedding~\cite[Equation (2.2)]{HR05} of partition algebras $\CC A_k(\xi)\subset \CC A_{k+\frac{1}{2}}(\xi)$.

\begin{theorem}\label{thm:indexing_set}
	For $k\in\ZZ_{> 0}$, the irreducible representations of $\CC I_k$ and $\CC I_{k+\frac{1}{2}}$ are indexed by the elements 
	of $\widehat{I}_k:=\Lambda_{\leq k}\setminus\{\emptyset\}$ and $\widehat{I}_{k+\frac{1}{2}}:=\Lambda_{\leq k}$,  respectively. 
\end{theorem}
\begin{proof}
	Choose $n\in\ZZ_{> 0}$ such that $n\geq k$. By applying Corollary~\ref{prop:brat_rook}, we see that in the decomposition of $V^{\otimes k}$, only those irreducible representations of $\CC R_n$ appear which are indexed by the elements of $\widehat{I}_k$. Since $n\geq k$, by Theorem~\ref{thm:swd_rook} we have
	$\CC I_k\cong \End_{\CC R_n}(V^{\otimes k})$. So, by the centralizer theorem~\cite[Theorem 5.4]{HR05}, the irreducible representations of $\CC I_k$ are indexed by the elements of $\widehat{I}_k$.

Similarly, by choosing $n\in\ZZ_{>0}$ such that $n\geq k+1$ and then applying Corollary~\ref{thm:swdfrook_half}, Proposition~\ref{prop:res} and the centralizer theorem, we get that the irreducible representations of $\CC I_{k+\frac{1}{2}}$ are indexed by the elements of $\widehat{I}_{k+\frac{1}{2}}$. 
\end{proof}

Suppose $I_{k}^{\lambda}$ and $I_{k+\frac{1}{2}}^{\mu}$ denote the irreducible representations indexed by $\lambda\in\widehat{I}_{k}$ and $\mu\in \widehat{I}_{k+\frac{1}{2}}$ of $\CC I_k$ and $\CC I_{k+\frac{1}{2}}$, respectively. 
\begin{theorem}\phantomsection
	\label{thm:decomposition}
	\begin{enumerate}[(i)]
		\item For $n\geq k$, as a $(\CC R_n,\CC I_k)$-bimodule we have $V^{\otimes k}\cong \bigoplus_{\lambda\in \widehat{I}_k}V^{\lambda}_n\otimes I^{\lambda}_k$.
\item For $n\geq k+1$, as a $(\CC R_{n-1},\CC I_{k+\frac{1}{2}})$-bimodule we have
		$V^{\otimes k}\cong \bigoplus_{\mu\in\widehat{I}_{k+\frac{1}{2}}} V^{\mu}_{n-1}\otimes I^{\mu}_{k+\frac{1}{2}}$.
\end{enumerate}
\end{theorem}
\begin{proof}
	The proof of the first part (respectively, the second part) is an application of Theorem~\ref{thm:swd_rook} (respectively, Corollary~\ref{thm:swdfrook_half}), Theorem~\ref{thm:indexing_set}, and the centralizer theorem.	
\end{proof}

{\bf{A tower, branching rule and the Bratteli diagram}}.
	For $k\in\ZZ_{>0}$, define the following embedding
	\begin{align}\label{eq:embedding}
	\eta_k:\CC I_{k}\to\CC I_{k+\frac{1}{2}},\quad \eta_{k}(x_{d})=x_{d'}, 
	\end{align}
	where, given $d\in I_k$, the element $d'\in I_{k+\frac{1}{2}}$ is obtained from $d$ by adding the block $\{(k+1),(k+1)'\}$. Corollary~\ref{coro:obprod} implies that $\eta_k$ is  an
	 algebra homomorphism. Using~\eqref{eq:embedding}, we have the following tower of totally propagating partition algebras:
\begin{equation}\label{eq:tower}
\CC I_{\frac{1}{2}}= \CC I_{1}\subset \CC I_{\frac{3}{2}}\subset \CC I_{2}\subset\cdots.
\end{equation}

 We state the following theorem from \cite[Theorem 5.9]{Ram90} in order to describe the branching rule for the embedding $\CC I_k\subset \CC I_{k+\frac{1}{2}}$ in Theorem~\ref{thm:item2}(i). The branching rule for $\CC I_{k+\frac{1}{2}}\subset \CC I_{k+1}$ is given in Theorem~\ref{thm:item2}(ii). 
\begin{theorem}\label{thm:ram}
	Let $A$ be a subalgebra of an algebra $B$ and let $M$ be a finite dimensional repesentation of $B$, which is a semisimple representation of both $B$ and $A$. Let $W$ and $V$ be representations of $B$ and $A$, respectively, both being subrepresentations of $M$. Then, the multiplicity of the representation $\Hom_{B}(M,W)$ of $\End_{B}(M)$ in the restriction of the representation $\Hom_{A}(M,V)$ of $\End_{A}(M)$ to $\End_{B}(M)$ is equal to the multiplicity of $V$ in the restriction of $W$ from $B$ to $A$.
\end{theorem}
\begin{theorem}[Branching rule] 
	\phantomsection
		\label{thm:item2}
		\begin{enumerate}[(i)]
	\item	For $\mu\in \widehat{I}_{k+\frac{1}{2}}$,
		\begin{align*}
	\Res_{\CC I_k}^{\CC I_{k+\frac{1}{2}}}I_{k+\frac{1}{2}}^{\mu}\cong\begin{cases}
I_{k}^{(1)} & \text{if } \mu=\emptyset,\\
 \bigoplus_{\nu\in\mu^{+,=}} I_{k}^{\nu} & \text{if }\mu\neq\emptyset.
	\end{cases}
		\end{align*}
\item  For $\lambda\in \widehat{I}_{k+1}$, we have
	$\Res_{\CC I_{k+\frac{1}{2}}}^{\CC I_{k+1}} I_{k+1}^{\lambda}\cong\bigoplus_{\mu\in\lambda^{-}} I^{\mu}_{k+\frac{1}{2}}$.
\end{enumerate}
\end{theorem}
\begin{proof}
\begin{enumerate}[(i)]
\item The result follows from Theroem~\ref{thm:decomposition} and Theorem~\ref{thm:ram}.
	
\item  Choose $n\in\ZZ_{> 0}$ such that $n\geq k+1$. Then,
\begin{align*}
&\Res^{\CC I_{k+1}}_{\CC I_{k+\frac{1}{2}}}I^{\lambda}_{k+1}\cong\Res^{\CC I_{k+1}}_{\CC I_{k+\frac{1}{2}}}\Hom_{\CC R_n}(V^{\otimes (k+1)},V^{\lambda}_n), \text{ by the part $(i)$ of Theorem~\ref{thm:decomposition}}\\
&\cong \Res^{\CC I_{k+1}}_{\CC I_{k+\frac{1}{2}}} \Hom_{\CC R_n}((\widehat{\Ind}^{\CC R_{n}}_{\CC R_{n-1}}\Res^{\CC R_{n}}_{\CC R_{n-1}} V^{\emptyset}_n)^{k+1},V^{\lambda}_n),\text{ by Proposition~\ref{prop:ind_res_tensor}}\\
&\cong \Res^{\CC I_{k+1}}_{\CC I_{k+\frac{1}{2}}} \Hom_{\CC R_{n-1}}((\widehat{\Ind}^{\CC R_{n}}_{\CC R_{n-1}}\Res^{\CC R_{n}}_{\CC R_{n-1}} V_{\emptyset})^{k},\widehat{\Res}^{\CC R_n}_{\CC R_{n-1}} V^{\lambda}_n), \text{ by Proposition~\ref{prop:induct_restrict}}\\
&\cong \Hom_{\CC R_{n-1}}(V^{\otimes k},\bigoplus_{\mu\in \lambda^{-}} V^{\mu}_{n-1}), \text{ by definition of $\widehat{\Res}^{\CC R_n}_{\CC R_{n-1}}(-)$}\\
&\cong\bigoplus_{\mu\in \lambda^{-}}\Hom_{\CC R_{n-1}}(V^{\otimes k},V^{\mu}_{n-1})
\cong\bigoplus_{\mu\in \lambda^{-}} I^{\mu}_{k+\frac{1}{2}}, \text{ by the part $(ii)$ of Theorem~\ref{thm:decomposition}}.     \tag*{\qedhere} 
\end{align*}
\end{enumerate}
\end{proof}

Using Thereom~\ref{thm:item2}, we get the Bratteli diagram $\widehat{I}$ for the tower~\eqref{eq:tower} in which the sets of vertices at level $k$ and at level $k+\frac{1}{2}$ are $\widehat{I}_{k}=\Lambda_{\leq k}\setminus\{\emptyset\}$ and  $\widehat{I}_{k+\frac{1}{2}}=\Lambda_{\leq k}$, respectively. For $\mu\in\widehat{I}_{k}$ and $\nu\in \widehat{I}_{k+\frac{1}{2}}$, there is an edge between 
 $\mu$ and $\nu$ if and only if $\nu=\mu$ or $\nu\in\mu^-$.
 For $\nu\in \widehat{I}_{k+\frac{1}{2}}$ and $\lambda\in \widehat{I}_{k+1}$, there is an edge between $\nu$ and $\lambda$ if and only if $\lambda\in\nu^+$.

\begin{example} 
The Bratteli diagram for the tower~\eqref{eq:tower} up to level $3$ is illustrated in Figure \ref{fig:BDTPPA}.
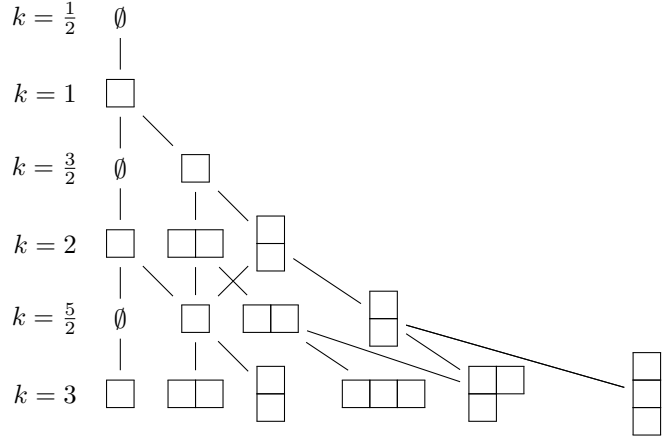
\begin{figure}[H]\centering
	\begin{tikzpicture}[scale=.5]
	\node (a) at (-2,0) {$k=\frac{1}{2}$};
	\node (b) at (-2,-2) {$k=1$};
	\node (c) at (-2,-4) {$k=\frac{3}{2}$};
	\node (d) at (-2,-6) {$k=2$};
	\node (e) at (-2,-8) {$k=\frac{5}{2}$};
	\node (f) at (-2,-10) {$k=3$};
	\node (1) at (0,0) {$\emptyset$};
	\node (2) at (0,-2) {$\ytableausetup{boxsize=1em} \begin{ytableau}
                                                        \empty \\
                                                       \end{ytableau}$};
	\node (3) at (0,-4) {$\emptyset$};
	\node (4) at (2,-4) {$\ytableausetup{boxsize=1em} \begin{ytableau}
                                                        \empty \\
                                                       \end{ytableau}$};
	\node (5) at (0,-6) {$\ytableausetup{boxsize=1em} \begin{ytableau}
                                                        \empty \\
                                                       \end{ytableau}$};
	\node (6) at (2,-6) {$\ytableausetup{boxsize=1em} \begin{ytableau}
                                                        \empty & \empty \\
                                                       \end{ytableau}$};
	\node (7) at (4,-6) {$\ytableausetup{boxsize=1em} \begin{ytableau}
                                                        \empty \\
                                                        \empty \\
                                                       \end{ytableau}$};
	\node (8) at (0,-8) {$\emptyset$};
	\node (9) at (2,-8) {$\ytableausetup{boxsize=1em} \begin{ytableau}
                                                        \empty \\
                                                       \end{ytableau}$};
	\node (10) at (4,-8) {$\ytableausetup{boxsize=1em} \begin{ytableau}
                                                        \empty & \empty \\
                                                       \end{ytableau}$};
	\node (11) at (7,-8) {$\ytableausetup{boxsize=1em} \begin{ytableau}
                                                        \empty \\
                                                        \empty \\
                                                       \end{ytableau}$};
	\node (12) at (0,-10) {$\ytableausetup{boxsize=1em} \begin{ytableau}
                                                        \empty \\
                                                       \end{ytableau}$};
	\node (13) at (2,-10) {$\ytableausetup{boxsize=1em} \begin{ytableau}
                                                        \empty & \empty\\
                                                       \end{ytableau}$};
	\node (14) at (4,-10) {$\ytableausetup{boxsize=1em} \begin{ytableau}
                                                        \empty \\
                                                        \empty \\
                                                       \end{ytableau}$};
	\node (15) at (7,-10) {$\ytableausetup{boxsize=1em} \begin{ytableau}
                                                        \empty & \empty & \empty \\
                                                       \end{ytableau}$};
	\node (16) at (10,-10) {$\ytableausetup{boxsize=1em} \begin{ytableau}
                                                        \empty & \empty \\
                                                        \empty \\
                                                       \end{ytableau}$};
	\node (17) at (14,-10) {$\ytableausetup{boxsize=1em} \begin{ytableau}
                                                        \empty \\
                                                        \empty \\
                                                        \empty \\
                                                       \end{ytableau}$};
	
	\draw (1) -- (2) -- (3) -- (5) -- (8) -- (12); 
	\draw (2) -- (4) -- (6) -- (9) -- (13);
	\draw (4) -- (7) -- (11) -- (17);
	\draw (5) -- (9);
	\draw (6) -- (10) -- (15);
	\draw (7) -- (9) -- (14);
	\draw (10) -- (16);
	\draw (11) -- (16);
	\draw (11) -- (17);
	\end{tikzpicture}
	\caption{Bratteli diagram, up to level 3, for the tower of totally propagating partition algebras}
	\label{fig:BDTPPA}
\end{figure}
\end{example}

\begin{corollary} For $t\in \frac{1}{2}\ZZ_{> 0}$, the dimension of the irreducible representation $I^{\lambda}_t$ of $\CC I_{t}$ is the number of paths (see Convention~\ref{con} for the definition of a path)
	from $\emptyset\in \widehat{I}_{\frac{1}{2}}$ to
	$\lambda\in \widehat{I}_t$ in the Bratteli diagram $\widehat{I}$.
\end{corollary}

\section{Jucys--Murphy elements of totally propagating partition algebras}
\label{sec:JM}
We begin this section with some definitions and notation  which are needed to define Jucys--Murphy elements of $\CC I_t$, for $t\in\frac{1}{2}\ZZ_{> 0}$. 

 Let $k\in\ZZ_{> 0}$. Suppose that $\mathcal{U}_{k}$ denotes the collection of the set partitions of $\{1,2,\ldots,k\}$. For $\mathcal{P}=\{B_1,\ldots,B_l\}\in \mathcal{U}_k$, let $d_{\mathcal{P}}\in I_{k}$ be the partition diagram corresponding to the set partition
$\{B_1\cup B_1',\ldots,B_l\cup B_l' \}$; for any subset $S$ of $\{1,2,\ldots,k\}$, let $S'$ denote the set $\{i'\in\{1',2',\ldots,k'\}\mid i\in S\}$. For $k\geq 2$, suppose that $\mathcal{S}_{k}$ denotes the collection of the set partitions of $\{1,2,\ldots,k\}$ with at least two blocks.
 For $\mathcal{P}\in\mathcal{S}_k$ and distinct blocks $C,D\in \mathcal{P}$, let 
 $d_{\mathcal{P},\{C,D\}}\in I_k$ be the partition diagram whose blocks are 
 $(a)$ $B\cup B'$ for $B\in \mathcal{P}$ and $B\notin\{C,D\}$, and  $(b)$ $C\cup D'$, $D\cup C'$.

\begin{example}\label{ex:example1}
	For $k=5$, let $\mathcal{P}=\{\{1,3\},\{4\},\{2,5\}\}$ be a set partition of $\{1,2,3,4,5\}$. Then $d_{\mathcal{P},\{\{1,3\},\{2,5\}\}}\in I_{5}$, depicted in Figure~\ref{fig:db}, corresponds to $\{\{1,3,2',5'\},\{4,4'\},\{2,5,1',3'\}\}$. 
	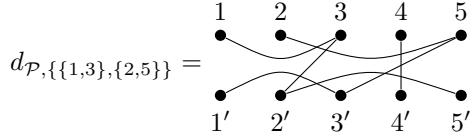
\begin{figure}[H]\centering
		\begin{tikzpicture}
		[scale=0.8,
		mycirc/.style={circle,fill=black, minimum size=0.1mm, inner sep = 1.5pt}]
		\node (1) at (-1.9,0.5) {$d_{\mathcal{P},\{\{1,3\},\{2,5\}\}} =$};
		\node[mycirc,label=above:{$1$}] (n1) at (0,1) {};
		\node[mycirc,label=above:{$2$}] (n2) at (1,1) {};
		\node[mycirc,label=above:{$3$}] (n3) at (2,1) {};
		\node[mycirc,label=above:{$4$}] (n4) at (3,1) {};
    	\node[mycirc,label=above:{$5$}] (n5) at (4,1) {};
		\node[mycirc,label=below:{$1'$}] (n1') at (0,0) {};
		\node[mycirc,label=below:{$2'$}] (n2') at (1,0) {};
		\node[mycirc,label=below:{$3'$}] (n3') at (2,0) {}; 
		\node[mycirc,label=below:{$4'$}] (n4') at (3,0) {}; 
		\node[mycirc,label=below:{$5'$}] (n5') at (4,0) {};
		
		\draw (n1)..controls(1,0.5)..(n3);
		\draw (n2)..controls(2.5,0.5)..(n5);
		\draw (n4)--(n4');
		\draw (n3)--(n2');
		\draw (n3')--(n5);
		\draw (n1')..controls(1,0.5)..(n3');
		\draw (n2')..controls(2.5,0.5)..(n5');
		\end{tikzpicture}
		\caption{An example of $d_{\mathcal{P},\{C,D\}}$}
		\label{fig:db}
		\end{figure}
\end{example}

  Let $|F|$ denote the cardinality of a finite set $F$.  Define the following elements in $\CC I_{k}$:
\begin{align}
Z_{k}&:=\sum_{\mathcal{P}\in \mathcal{U}_k}|\mathcal{P}|x_{d_{\mathcal{P}}},\label{eq:JM_Z_k}\\
\tilde{Z}_1:=0 \text { and for $k\geq 2$}, \tilde{Z}_k&:=\sum_{\mathcal{P}\in\mathcal{S}_{k}}\bigg(\sum_{\substack{C,D\in{\mathcal{P}},\\C\neq D}}x_{d_{\mathcal{P},\{C,D\}}}\bigg).\label{eq:tildeJM_Z_k}
\end{align}

 For $k\geq 2$, let $\mathcal{S}_{k+\frac{1}{2}}(\subset \mathcal{S}_{k+1} )$ be the collection of the set partitions of $\{1,2,\ldots,k+1\}$ with at least three blocks. Throughout this section, for $\mathcal{P}\in\mathcal{S}_{k+1}$, let $B^{\mathcal{P}}\in\mathcal{P}$ denote the block containing $k+1$. For the distinct blocks $C\neq B^{\mathcal{P}}$ and $D\neq B^{\mathcal{P}}$ in $\mathcal{P}\in S_{k+\frac{1}{2}}$, define $\tilde{d}_{\mathcal{P},\{C,D\}}\in I_{k+\frac{1}{2}}$ whose blocks are $(a)$ $B\cup B'$, for $B\in \mathcal{P}$ and $B\notin\{C,D\}$ and $(b)$ $C\cup D'$, $D\cup C'$.

 \begin{example}
 For $k=4$, let $\mathcal{P}=\{\{1,3\},\{4\},\{2,5\}\}$ be a set partition of $\{1,2,3,4,5\}$. Here $B^{\mathcal{P}}=\{2,5\}$. Then  $\tilde{d}_{\mathcal{P},\{\{1,3\},\{4\}\}}\in I_{4+\frac{1}{2}}$, depicted in Figure~\ref{fig:dbt}, corresponds to  $\{\{1,3,4'\},\{4,1',3'\},\{2,5,2',5'\}\}$.
 
 	\begin{figure}[ht]\centering
 	\begin{tikzpicture}
 	[scale=0.8,
 	mycirc/.style={circle,fill=black, minimum size=0.1mm, inner sep = 1.5pt}]
 	\node (1) at (-1.9,0.5) {$\tilde{d}_{\mathcal{P},\{\{1,3\},\{4\}\}} =$};
 	\node[mycirc,label=above:{$1$}] (n1) at (0,1) {};
 	\node[mycirc,label=above:{$2$}] (n2) at (1,1) {};
 	\node[mycirc,label=above:{$3$}] (n3) at (2,1) {};
 	\node[mycirc,label=above:{$4$}] (n4) at (3,1) {};
 	\node[mycirc,label=above:{$5$}] (n5) at (4,1) {};
 	\node[mycirc,label=below:{$1'$}] (n1') at (0,0) {};
 	\node[mycirc,label=below:{$2'$}] (n2') at (1,0) {};
 	\node[mycirc,label=below:{$3'$}] (n3') at (2,0) {}; 
 	\node[mycirc,label=below:{$4'$}] (n4') at (3,0) {}; 
 	\node[mycirc,label=below:{$5'$}] (n5') at (4,0) {};
 	
 	\draw (n1)..controls(1,0.5)..(n3);
 	\draw (n3)--(n4');
 	\draw (n2)..controls(2.5,0.5)..(n5);
 	\draw (n3')--(n4);
 	\draw (n5')--(n5);
 	\draw (n1')..controls(1,0.5)..(n3');
 	\draw (n2')..controls(2.5,0.5)..(n5');
 \end{tikzpicture} 
 \caption{An example of $\tilde{d}_{\mathcal{P},\{C,D\}}$}
 \label{fig:dbt}
 \end{figure}
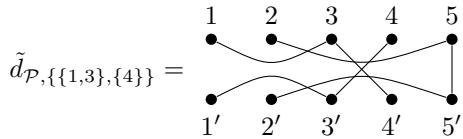
 
 \end{example}

Let $Z_{\frac{1}{2}}=1$ and $\tilde{Z}_{\frac{1}{2}}=0$. For $k\geq 1$, we define the following elements in $\CC I_{k+\frac{1}{2}}$:
\begin{align}
 &Z_{k+\frac{1}{2}}:= \sum_{\mathcal{P}\in \mathcal{S}_{k+1}}(|\mathcal{P}|-1)x_{d_{\mathcal{P}}} = 
 Z_{k}+\sum_{\substack{\mathcal{P}\in\mathcal{S}_{k+1},\label{eq:JM_new_Z_k}\\|B^{\mathcal{P}}| >1}}(|\mathcal{P}|-1)x_{d_{\mathcal{P}}},\\
 &\tilde{Z}_{\frac{3}{2}}:=0, \text{ and for $k\geq 2$},\nonumber\\ 
&\tilde{Z}_{k+\frac{1}{2}}:= \sum_{\mathcal{P}\in \mathcal{S}_{k+\frac{1}{2}}}\bigg(\sum_{\substack{C\neq B^{\mathcal{P}},
D\neq B^{\mathcal{P}},\\ C\neq D}} x_{\tilde{d}_{\mathcal{P},\{C,D\}}}\bigg)
= \tilde{Z}_{k}+ \sum_{\mathcal{P}\in \mathcal{S}_{k+\frac{1}{2}}}\bigg(\sum_{\substack{|B^{\mathcal{P}}|>1,C\neq B^{\mathcal{P}},\\
		D\neq B^{\mathcal{P}}, C\neq D}} x_{\tilde{d}_{\mathcal{P},\{C,D\}}}\bigg).\label{eq:JM_new_Z_k_1}
\end{align}
In \eqref{eq:JM_new_Z_k} and \eqref{eq:JM_new_Z_k_1}, $Z_{k}$  and $\tilde{Z}_{k}$ are viewed in $\CC I_{k+\frac{1}{2}}$ using the embedding~\eqref{eq:embedding} $\eta_{k}:\CC I_{k}\to \CC I_{k+\frac{1}{2}}$.

  For $\lambda\in \Lambda_{\leq n}$, we write $b\in\lambda$ to mean that $b$ is a box in $\lambda$. For a vector space $W$, let ${\rm id}_{W}$ denote the identity operator on $W$. 
	  \begin{theorem}\phantomsection
	  	\label{thm:JMZ}
	  	\begin{enumerate}[(i)]
	  \item	If $\lambda\in \Lambda_{\leq n}$, then as operators on $V^{\lambda}_n$
	  	\begin{displaymath}
	  	\kappa_{n}=|\lambda|{\rm id}_{V^{\lambda}_n} \text{ and } \tilde{\kappa}_{n}=\sum_{b\in\lambda}\ct(b) {\rm id}_{V^{\lambda}_n}.
	  	\end{displaymath}
	  	
	  	\item As operators on $V^{\otimes k}$, we have $Z_{k}=\kappa_{n}, Z_{k+\frac{1}{2}}=\kappa_{n-1}$, and also as operators on 
	  	$V^{\otimes (k+\frac{1}{2})}$, we have $\tilde{Z}_{k}=\tilde{\kappa}_{n},\tilde{Z}_{k+\frac{1}{2}}=\tilde{\kappa}_{n-1}$.

	  \item	 
	  	\begin{enumerate}
	  		\item For $\lambda\in \widehat{I}_{k}$, we have 
	  		$Z_{k}=|\lambda|{\rm id}_{I^{\lambda}_{k}} \text{ and }$ 
	  		$ \tilde{Z}_{k}=\sum_{b\in\lambda}\ct(b) {\rm id}_{I^{\lambda}_{k}}$,  as operators on $I^{\lambda}_{k}$.
	  		\item For $\mu\in \widehat{I}_{k+\frac{1}{2}}$, we have 
	  		$Z_{k+\frac{1}{2}}=|\lambda|{\rm id}_{I^{\lambda}_{k+\frac{1}{2}}}  \text{ and }  \tilde{Z}_{k+\frac{1}{2}}=\sum_{b\in\lambda}\ct(b){\rm id}_{I^{\lambda}_{k+\frac{1}{2}}},$
	  		as operators on $I^{\lambda}_{k+\frac{1}{2}}$.
	 	\end{enumerate}
	  	\end{enumerate}
	  \end{theorem}

		    \begin{proof}
		    	\begin{enumerate}[(i)]
		    		\item This follows from the definitions of $\kappa_n, \tilde{\kappa}_n$ and Theorem~\ref{thm:action}.
		\item Let $v:=e_{j_{1}}\otimes e_{j_{2}}\otimes\cdots\otimes e_{j_{k}}\in V^{\otimes k}$. Let $\underline{j}$ denote the multiset $\{j_1,\ldots,j_k\}$. Then,
		\begin{align}\label{al:X}
		X_{i}v_= \begin{cases} 
	v &\text{ if } i\in\underline{j},\\
		0                  &\text{ otherwise}.  
		\end{cases}
		\end{align}

	For $1\leq i\leq n$, let $B_{i}=\{s\in\{1,2,\ldots,k\}\mid j_s=i\}$. Let $\{m_1,m_2,\ldots,m_l\}_{\underline{j}}$ be the set of all the distinct elements of the multiset $\underline{j}$. Then the element $e_{j_{1}}\otimes e_{j_{2}}\otimes\cdots\otimes e_{j_{k}}$ determines 
	a set partition $\mathcal{P}_0=\{B_{m_1}, B_{m_{2}},\ldots, B_{m_l}\}$ of $\{1,2,\ldots,k\}$.
		 Then using~\eqref{al:X}, we get
		\begin{equation*}
		\label{eq:kappan}
		\kappa_nv= (X_{m_{1}}+X_{m_{2}}+\cdots+X_{m_{l}})v=|\mathcal{P}_0|v.
		\end{equation*}
		
	Using~\eqref{al:orbit},	every term, except $|\mathcal{P}_0|x_{d_{\mathcal{P}_0}}$, in the expression~\eqref{eq:JM_Z_k} of $Z_k$ acts 
	 on $v$ as zero, and 
	 $|\mathcal{P}_0|x_{d_{\mathcal{P}_0}}v=|\mathcal{P}_0|v$. Therefore, $Z_k=\kappa_n$ as operators on $V^{\otimes k}.$

	 For $2\leq i\leq n$, let $\mathcal{M}_{i,\underline{j}}=\{1,2,\ldots,i-1\}\cap \{m_1,m_2,\ldots,m_l\}_{\underline{j}}$.
	From the definition of $\tilde{X}_2$, we compute,
		\begin{align*}
	\tilde{X}_{2}v=
	\begin{cases}
	0 & \mbox{if } 2\notin\underline{j},\\
	0 & \mbox{if } 2\in\underline{j}  \text{ and } \mathcal{M}_{i,\underline{j}}=\emptyset, \\ 
	(1,2)v & \mbox{if } 2\in\underline{j}\mbox{ and } \mathcal{M}_{i,\underline{j}}=\{1\}.
	\end{cases}
	\end{align*}
Now using $\tilde{X}_3=s_2\tilde{X}_2s_2+s_2-s_2\gamma_2-\gamma_2s_2+Q_3$ and the above action of $\tilde{X}_2$, we get
\begin{align*}
	\tilde{X}_{3}v=
	\begin{cases}
	0 & \mbox{if } 3\notin\underline{j},\\
	0 & \mbox{if } 3\in\underline{j}  \text{ and } \mathcal{M}_{i,\underline{j}}=\emptyset, \\ 
	(1,3)v & \mbox{if } 3\in\underline{j}\mbox{ and } \mathcal{M}_{i,\underline{j}}=\{1\},\\
	(2,3)v & \mbox{if } 3\in\underline{j}\mbox{ and } \mathcal{M}_{i,\underline{j}}=\{2\},\\
	((1,3)+(2,3))v & \mbox{if } 3\in\underline{j}\mbox{ and } \mathcal{M}_{i,\underline{j}}=\{1,2\}.\\
	\end{cases}
	\end{align*}
		 For arbitrary $2\leq i \leq n$, using the definition~\eqref{eq:New_JM} of $\tilde{X}_{i}$ and induction on $i$, we obtain: 
		\begin{align}\label{al:tildeX}
	\tilde{X}_{i}v=
		\begin{cases}
		0 & \mbox{if } i\notin\underline{j},\\
		0 & \mbox{if } i\in\underline{j}  \text{ and } \mathcal{M}_{i,\underline{j}}=\emptyset, \\ 
	    ((t_1,i)+\cdots+(t_s,i))v & \mbox{if } i\in\underline{j}\mbox{ and } \mathcal{M}_{i,\underline{j}}= \{t_1,\ldots,t_s\}.
		\end{cases}
		\end{align}
 Then, from~\eqref{al:tildeX}, we have
	\begin{align*}
	\tilde{\kappa}_nv=(\tilde{X}_{m_1}+\cdots+\tilde{X}_{m_l})v=
	\begin{cases}
	\underset{1\leq p<q\leq l}{\sum}(m_p,m_q)v & \mbox{ if }  l\geq 2,\\
	0 & \mbox{ otherwise}.
	\end{cases}
	\end{align*}	
 Using~\eqref{al:orbit} and the expression~\eqref{eq:tildeJM_Z_k} of $\tilde{Z}_k$, we observe that (i) when $l=1$, $\tilde{Z}_kv=0$ and (ii) for $l\geq 2$,
	\begin{align}\label{eq:compute_Z}
	\tilde{Z}_{k}v&= \sum_{\mathcal{P}\in\mathcal{S}_{k}}\sum_{\substack{C,D\in\mathcal{P},\\C\neq D}}x_{d_{\mathcal{P},\{C,D\}}}v  
	=\sum_{1\leq p< q\leq l} x_{d_{{\mathcal{P}_0},\{B_{m_p},B_{m_q}\}}}v= \sum_{1\leq p< q\leq l} (m_p,m_q)v .
	\end{align}
Hence, $\tilde{Z}_{k}=\tilde{\kappa}_n$ as operators on $V^{\otimes k}$.	

The vector $v\otimes e_n\in V^{\otimes (k+\frac{1}{2})}$ gives a multiset $\underline{j}\cup\{n\}=\{j_1,j_2,\ldots,j_k,j_{k+1}\}$, where $j_{k+1}=n$. For $1\leq i\leq n$, define $C_{i}=\{ s\in\{1,2,\ldots,k+1\}\mid j_s=i\}$. Note that $k+1\in C_n$.
 Let $\{n_1,\ldots,n_{r},n\}$ be the set of all the distinct elements of $\underline{j}\cup\{n\}$. 
Then the set partition $\mathcal{P}_1$, determined by $v\otimes e_n$, of $\{1,2,\ldots,k+1\}$ is $\{C_{n_1},\ldots, C_{n_r},C_n\}$.
 We have
 $\kappa_{n-1}(v\otimes e_n)=(|\mathcal{P}_1|-1)(v\otimes e_n)$ and \begin{align*}
 \tilde{\kappa}_{n-1}(v\otimes e_n)=\begin{cases}
 \underset{\substack{1\leq p<q\leq r}}{\sum}(n_p,n_q)(v\otimes e_n) & \text{ if } r\geq 2,\\
 0 & \text{ otherwise}. 
 \end{cases}
 \end{align*}

	Then the proof of $Z_{k+\frac{1}{2}}=\kappa_{n-1}$ and $\tilde{Z}_{k+\frac{1}{2}}=\tilde{\kappa}_{n-1}$, as operators on $V^{\otimes (k+\frac{1}{2})}$, can be  given analogously by using the expressions of $Z_{k+\frac{1}{2}}$ and $\tilde{Z}_{k+\frac{1}{2}}$ given in \eqref{eq:JM_new_Z_k} and \eqref{eq:JM_new_Z_k_1},  respectively , and by
	applying the similar arguments as in the above discussion.
	
		\item This part follows from parts $(i)$, $(ii)$, Theorem~\ref{thm:swd_rook} and Corolloary~\ref{thm:swdfrook_half}. \qedhere
	\end{enumerate}
\end{proof}

\begin{corollary}\label{coro:center}
	For $t\in \frac{1}{2}\ZZ_{> 0}$, the elements $Z_{t}$ and $\tilde{Z}_{t}$ are in the center of 
	$\CC I_{t}$.
\end{corollary}

\begin{proof}
The result follows from Theorem~\eqref{thm:a}, Theorem~\ref{thm:swd_rook}, Corollary~\ref{thm:swdfrook_half}, and Theorem~\ref{thm:JMZ}(ii).
\end{proof}

For $t\in\frac{1}{2}\mathbb{Z}_{> 0}$, define the Jucys--Murphy elements of $\CC I_{t}$ as follows:
\begin{equation}\label{eq:JM_tppa}
\begin{aligned}
M_{\frac{1}{2}}&=1, \quad \tilde{M}_{\frac{1}{2}}=0,\\
M_{y}&=Z_{y}-Z_{y-\frac{1}{2}} \text{ and } \tilde{M}_{y}=\tilde{Z}_{y}-\tilde{Z}_{y-\frac{1}{2}}, \text{ for } y\in \frac{1}{2}\ZZ_{>0} \text{ and } \frac{1}{2}<y\leq t.
\end{aligned}
\end{equation}
In the following theorem, we show that the elements~\eqref{eq:JM_tppa} satisfy the fundamental properties of Jucys--Murphy elements as discussed in Section~\ref{sec:intro}. 
\begin{theorem}\label{thm:JM}
	Let $t\in\frac{1}{2}\mathbb{Z}_{> 0}$.
	\begin{enumerate}[(i)]
		\item The elements $M_{\frac{1}{2}},M_{1},\ldots,M_{t-\frac{1}{2}}, M_{t},\tilde{M}_{\frac{1}{2}},\tilde{M}_{1},\ldots,\tilde{M}_{t-\frac{1}{2}}, \tilde{M}_{t}$ 
		commute with each other in $\mathbb{C}I_{t}$.  Also, $\sum_{y\in\frac{1}{2}\ZZ_{>0},y\leq t} M_{y}$ and $\sum_{y\in\frac{1}{2}\ZZ_{>0},y\leq t}\tilde{M}_y$ are central elements of $\CC I_t$.
		\item  Let $\widehat{I}^{\mu}_{t}$ denote the set of paths (see Convention~\ref{con} for the definition of a path) from $\emptyset\in \widehat{I}_{\frac{1}{2}}$ to $\mu\in\widehat{I}_{t}$ in the Bratteli diagram $\widehat{I}$.
		Then there is a unique, up to scalars,  basis
		$\{v_{\gamma}\,|\, \gamma\in\widehat{I}^{\mu}_{t}\}$ of the irreducible representation $I^{\mu}_{t}$ of $\CC I_{t}$ such that, 
		 for $\gamma=(\gamma^{(\frac{1}{2})}=\emptyset, \gamma^{(1)},\gamma^{(\frac{3}{2})},\ldots,\gamma^{(t-\frac{1}{2})},\gamma^{(t)}=\mu)\in\widehat{I}^{\mu}_{t}$, and $l\in\mathbb{Z}_{> 0}$, $l\leq t$, we have
		\begin{align*}
	&\tilde{M}_{l}v_{\gamma}= \ct(\gamma^{(l)}/\gamma^{(l-\frac{1}{2})})v_{\gamma}, \quad  M_{l}v_{\gamma}= v_{\gamma},\\
	&\tilde{M}_{\frac{1}{2}}v_{\gamma}=0, M_{\frac{1}{2}}v_{\gamma}=v_{\gamma},	\end{align*}
	\begin{align*}
	\text{ and for $l>1$},	\quad	&\tilde{M}_{l-\frac{1}{2}}v_{\gamma}=\begin{cases} -\ct(\gamma^{(l-1)}/\gamma^{(l-\frac{1}{2})})v_{\gamma} & \text{ if } \gamma^{(l-1)}/\gamma^{(l-\frac{1}{2})}=\ytableausetup{boxsize=0.7em}
\begin{ytableau}
\empty \\
\end{ytableau},\\ 
		0    & \text{ if } \gamma^{(l-1)}=\gamma^{(l-\frac{1}{2})},
		\end{cases}	\\	
&M_{l-\frac{1}{2}}v_{\gamma}=\begin{cases} -v_{\gamma} & \text{ if } \gamma^{(l-1)}/\gamma^{(l-\frac{1}{2})}=\ytableausetup{boxsize=0.7em}
\begin{ytableau}
\empty \\
\end{ytableau},\\ 
0    & \text{ if } \gamma^{(l-1)}=\gamma^{(l-\frac{1}{2})}.
\end{cases}
\end{align*}
\end{enumerate}
\end{theorem}

\begin{proof}
\begin{enumerate}[(i)]
	\item Using the tower $\mathbb{C}I_{\frac{1}{2}}\subseteq  \mathbb{C}I_{1}\subset\cdots\subset\mathbb{C}I_{t}$
	we can view $Z_{\frac{1}{2}},\ldots,Z_{t},\tilde{Z}_{\frac{1}{2}},\ldots,\tilde{Z}_{t}$ as elements of $\mathbb{C}I_{t}$. By Corollary~\ref{coro:center}, both $Z_{y}$ and $\tilde{Z}_{y}$ are in the center of $\mathbb{C}I_{y}$, for $y\in \frac{1}{2}\ZZ_{> 0}$ and $y\leq t$. So, the Jucys--Murphy elements of $\mathbb{C}I_{t}$ commute with each other.  Both, $\sum_{y\in\frac{1}{2}\ZZ_{>0},y\leq t}M_{y} = Z_t$ and  $\sum_{y\in\frac{1}{2}\ZZ_{>0},y\leq t}\tilde{M}_y = \tilde{Z}_{t}$, are in the center of $\mathbb{C}I_{t}$. 
	
	\item We construct a basis $\{v_{\gamma}\mid \gamma\in \widehat{I}^{\mu}_{t}\}$ of $I^{\mu}_{t}$ inductively and then prove its desired properties. For $t=\frac{1}{2}$ or $1$, the algebra $\CC I_{t}$ is one dimensional and so there is only one irreducible representation of $\CC I_t$. In particular, $\dim I^{\emptyset}_{\frac{1}{2}}=1=\dim I^{(1)}_{1}$, thus
	there are unique choices of bases, up to scalars, of representations $  I^{\emptyset}_{\frac{1}{2}}$ and $ I^{(1)}_{1}$. Now we describe the inductive step. For $t >1$ and $\mu \in \widehat{I}_{t}$, by Theorem ~\ref{thm:item2}, an irreducible representation of $\mathbb{C}I_{t-\frac{1}{2}}$ in $\text{Res}^{\mathbb{C}I_{t}}_{\mathbb{C}I_{t-\frac{1}{2}}}I^{\mu}_{t}$ can occur at most once. By induction, we can choose a basis of each irreducible representation of ${\mathbb{C}I_{t-\frac{1}{2}}}$. For each irreducible representation $I^{\nu}_{t-\frac{1}{2}}$ of $\mathbb{C}I_{t-\frac{1}{2}}$ that appears in $\text{Res}^{\mathbb{C}I_{t}}_{\mathbb{C}I_{t-\frac{1}{2}}}I^{\mu}_{t}$, there is an edge between $\nu$ and $\mu$ in $\widehat{I}$. Taking the union of the bases of all $I^{\nu}_{t-\frac{1}{2}}$ appearing in $\text{Res}^{\mathbb{C}I_{t}}_{\mathbb{C}I_{t-\frac{1}{2}}}I^{\mu}_{t}$, we get a basis for $I^{\mu}_{t}$. 
	
	Now, we give the proof of the action of Jucys--Murphy elements on $v_{\gamma}$ for $\gamma\in \widehat{I}^{\mu}_{t}$. For $\gamma=(\gamma^{(\frac{1}{2})}, \gamma^{(1)},\gamma^{(\frac{3}{2})},\ldots,\gamma^{(t-\frac{1}{2})},\gamma^{(t)})\in\widehat{I}^{\mu}_{t}$, we have $\gamma^{(l)}/\gamma^{(l-\frac{1}{2})}= \ytableausetup{boxsize=0.7em}
	 \begin{ytableau}
	 \empty \\
	 \end{ytableau}$, where $l\in\ZZ_{> 0}$ and $l\leq t$. By the inductive construction of the basis above, $v_{\gamma}\in I^{\gamma^{(l)}}_{l}$ and also $v_{\gamma}\in I^{\gamma^{(l-\frac{1}{2})}}_{l-\frac{1}{2}}$.
	Then by Theorem~\ref{thm:JMZ},
	\begin{align*}
	M_{l}v_{\gamma}&=Z_{l}v_{\gamma}-Z_{l-\frac{1}{2}}v_{\gamma} =(|\gamma^{(l)}|-|\gamma^{(l-\frac{1}{2})}|)v_{\gamma}=v_{\gamma},  \text{  and}\\
	 \tilde{M}_{l}v_{\gamma}&=\tilde{Z}_{l}v_{\gamma}-\tilde{Z}_{l-\frac{1}{2}}v_{\gamma}
	  =\bigg(\sum_{b\in\gamma^{(l)}}\ct(b)-\sum_{b'\in\gamma^{(l-\frac{1}{2})}}\ct(b')\bigg)v_{\gamma}=\ct(\gamma^{(l)}/\gamma^{(l-\frac{1}{2})})v_{\gamma}.
	\end{align*}
	
	For $2\leq l$, either $\gamma^{(l-1)}/\gamma^{(l-\frac{1}{2})}=\ytableausetup{boxsize=0.7em}
	\begin{ytableau}
	\empty \\
	\end{ytableau}$ or $\gamma^{(l-1)}=\gamma^{(l-\frac{1}{2})} $. Again by Theorem~\ref{thm:JMZ}, we have
	\begin{align*}
	&M_{l-\frac{1}{2}}v_{\gamma}=Z_{l-\frac{1}{2}}v_{\gamma}-Z_{l-1}v_{\gamma}
	 =(|\gamma^{(l-\frac{1}{2})}|- |\gamma^{(l-1)}|)v_{\gamma}= \begin{cases}
	 -v_{\gamma}&\text{if } \gamma^{(l-1)}/\gamma^{(l-\frac{1}{2})}=\ytableausetup{boxsize=0.7em}
	 \begin{ytableau}
	 \empty \\
	 \end{ytableau},\\
	 0 &\text{if } \gamma^{(l-1)}=\gamma^{(l-\frac{1}{2})},
	 \end{cases}\\
	&\tilde{M}_{l-\frac{1}{2}}v_{\gamma}=\tilde{Z}_{l-\frac{1}{2}}v_{\gamma}-\tilde{Z}_{l-1}v_{\gamma}
	= \bigg(\sum_{b'\in \gamma^{(l-\frac{1}{2})}} \ct(b')-\sum_{b''\in \gamma^{(l-1)}}\ct(b'')\bigg)v_{\gamma}\\&
	\qquad\qquad\qquad\qquad\qquad\quad\quad=
\begin{cases}
	-\ct(\gamma^{(l-1)}/\gamma^{(l-\frac{1}{2})})v_{\gamma} &\text{if } \gamma^{(l-1)}/\gamma^{(l-\frac{1}{2})}=\ytableausetup{boxsize=0.7em}
	\begin{ytableau}
	\empty \\
	\end{ytableau},\\
	 0 &\text{if } \gamma^{(l-1)}=\gamma^{(l-\frac{1}{2})}.
	\end{cases} 
	\end{align*}
	Thus, the basis $\{v_{\gamma}\,|\, \gamma\in\widehat{I}^{\mu}_{t}\}$ of $I^{\mu}_{t}$ consists of simultaneous eigenvectors of Jucys--Murphy elements of $\CC I_t$. For a path $\gamma \in \widehat{I}^{\mu}_{t}$, let $\alpha_\gamma = (a_{\frac{1}{2}, \gamma}, \tilde{a}_{\frac{1}{2}, \gamma}, a_{1, \gamma}, \tilde{a}_{1, \gamma}, \ldots, a_{t,\gamma}, \tilde{a}_{t,\gamma})$, where $M_y v_{\gamma} = a_{y, \gamma}v_\gamma$ and $\tilde{M}_y v_{\gamma} = \tilde{a}_{y, \gamma} v_\gamma$ for $y\in\frac{1}{2}\ZZ_{>0},y\leq t$. Given paths $\gamma_1, \gamma_2 \in \widehat{I}^{\mu}_{t}$, $\gamma_1 \neq \gamma_2$, we have $\alpha_{\gamma_1} \neq \alpha_{\gamma_2}$. Such a simultaneous eigenbasis is unique (up to scalars).
        \qedhere
\end{enumerate}
\end{proof}

The basis constructed here is the (canonical) Gelfand--Tsetlin basis of an irreducible representation $I^{\mu}_{t}$ of $\CC I_{t}$, for $t\in\frac{1}{2}\ZZ_{> 0}$. 
Recall that a Gelfand--Tsetlin vector is defined to be an element of Gelfand--Tsetlin basis of some irreducible representation of $\CC I_{t}$.
Now, the following corollary implies the spectral significance of Jucys--Murphy elements in the representation theory of totally propagating partition algebras.

\begin{corollary}\label{cor:jmtppa}
Let $l\in\ZZ_{> 0}$.	The actions of the elements $\tilde{M}_{y}$, for $y\in \frac{1}{2}Z_{> 0}$ and $y\leq l$, on Gelfand--Tsetlin vectors are sufficient to distinguish the nonisomorphic irreducible representations of $\CC I_{l}$. While in order to distinguish the nonisomorphic irreducible representations of $\CC I_{l-\frac{1}{2}}$, the actions of both $M_{y}$ and $\tilde{M}_{y}$ for  $y\in \frac{1}{2}Z_{> 0}$ and $y\leq l-\frac{1}{2}$ on Gelfand--Tsetlin vectors are needed to be considered.
\end{corollary}

  {\bf Acknowledegments}: The authors thank Prof. Arun Ram for encouragement and fruitful discussions. The authors also thank Prof. James East for pointing out the paper~\cite{JE} which studies the cellularity of totally propapgating partition algebras. S.S. thanks Prof. Upendra Kulkarni for his question on Schur--Weyl duality for the rook monoid algebra which was a motivation for this project. A.M. was supported by visiting professorship at UFPA, Brazil. S.S. was supported by the institute postdoctoral fellowship at TIFR, Mumbai, India.

{\footnotesize
\bibliographystyle{abbrvurl}
\bibliography{refs}
}

\Addresses

\end{document}